\documentclass[12pt,a4paper,reqno]{amsart}
\usepackage{fouriernc} % thicker font - some subscripts were lost

\usepackage{geometry}
\usepackage{xcolor}
\usepackage{graphicx}
\usepackage{placeins}
\usepackage[numbers]{natbib}

\newtheorem{corollary}{Corollary}
\newtheorem{theorem}{Theorem}
\newtheorem{conjecture}{Conjecture}
\newtheorem*{conjecture*}{Conjecture}
\newtheorem{lemma}{Lemma}
\theoremstyle{remark}
\newtheorem*{remark}{Remark}

\numberwithin{equation}{section}

\renewcommand{\Re}{\mathfrak{Re}}
\renewcommand{\Im}{\mathfrak{Im}}
\newcommand{\NN}{\mathbb{N}}
\newcommand{\E}{\mathbb{E}}

\begin{document}

\title[Integer moments of the derivatives of zeta]{Integer moments of the derivatives of the Riemann zeta function}

\author{Christopher Hughes}
\address{Department of Mathematics, University of York, York, YO10 5GH, United Kingdom}
\email{christopher.hughes@york.ac.uk}
\author{Andrew Pearce-Crump}
\address{School of Mathematics, Fry Building, Woodland Road, Bristol, BS8 1UG, United Kingdom}
\email{andrew.pearce-crump@bristol.ac.uk}

\begin{abstract}
    We conjecture the full asymptotic expansion of a product of Riemann zeta functions, evaluated at the non-trivial zeros of the zeta function, with shifts added in each argument. By taking derivatives with respect to these shifts, we form a conjecture for the integer moments of mixed derivatives of the zeta function. This generalises a result of the authors where they took complex moments of the first derivative of the zeta function, evaluated at the non-trivial zeros. We approach this problem in two different ways: the first uses a random matrix theory approach, and the second by the Ratios Conjecture of Conrey, Farmer, and Zirnbauer.
\end{abstract}

\maketitle

\section{Introduction}\label{sect:Intro}
A difficult problem in analytic number theory is to calculate the full asymptotics of the moments of the Riemann zeta function $\zeta (s)$. The type of moments considered in this paper are given as sums of products of the derivatives of the zeta function, evaluated at the non-trivial zeros of the zeta function, which we refer to as `moments of mixed derivatives'. We formulate the following conjecture for a shifted version of these moments by following approach taken in forming the Ratios Conjecture of Conrey, Farmer, and Zirnbauer \cite{CFZ05}.

\begin{conjecture}\label{RatioMom}
Assume the Riemann Hypothesis. Additionally assume that small shifts $\alpha_j$ satisfy $|\Re(\alpha_j)| < \frac{1}{4}$ and $|\Im(\alpha_j)| \ll_\varepsilon T^{1-\varepsilon}$ for every $\varepsilon > 0$, and that $|\delta|<1/4$.

The discrete moments of shifts of $\zeta (s)$, evaluated at the non-trivial zeros $\rho = \frac{1}{2} + i \gamma$ of $\zeta (s)$, are given by
\begin{multline}\label{eq:ZRatios}
\sum_{0 < \gamma \leq T}  \zeta \left( \frac{1}{2} + i \gamma + \alpha_1 \right)\dots\zeta \left( \frac{1}{2} + i \gamma + \alpha_k \right) = \\
 \left. \frac{d}{d \delta} \frac{1}{2 \pi} \int_{1}^{T} Z_{\alpha_1,\dots,\alpha_k,\delta} + \left( \frac{t}{2 \pi} \right)^{-\alpha_1 - \delta} Z_{-\delta, \alpha_2,\dots,\alpha_k,-\alpha_1} +\dots + \left( \frac{t}{2 \pi} \right)^{-\alpha_k - \delta} Z_{ \alpha_1,\dots,-\delta, -\alpha_k} dt  \right|_{\delta=0} \\
+ \frac{T}{2\pi} \log \frac{T}{2\pi} + O\left(T^{1/2 +\varepsilon}\right),
\end{multline}
where $Z_{\alpha_1,\dots,\alpha_k,\delta}$ is defined as the following quotient of zeta functions multiplied by an arithmetic factor
\[
 Z_{\alpha_1,\dots,\alpha_k,\delta} = \frac{  \zeta (1+\alpha_1 + \delta)  \dots \zeta (1+\alpha_k +\delta)}{\zeta (1+\alpha_1)  \dots \zeta (1+\alpha_k)} A_{\{\alpha_1,\dots,\alpha_k\}}(\delta).
\]
and where the arithmetic factor is given by
\begin{equation}\label{eq:A}
    A_{\{\alpha_1,\dots,\alpha_k\}}(\delta) := \\
    \prod_p \frac{ 1 + F_1(p) + F_2(p) + \dots + F_k(p)}{(1-p^{-(1+\alpha_1)})\dots(1-p^{-(1+\alpha_k)})}
\end{equation}
where for $1\leq m\leq k$
\[
F_m(p) = (-1)^m \sum_{\substack{J \subset\{1,\dots,k\} \\ |J|=m}} p^{-\left(m +(m-1)\delta + \sum_{j \in J} \alpha_j \right)}.
\]
\end{conjecture}

\begin{remark}\hfill
\begin{enumerate}
    \item Note that if $\delta=0$ then
    \[
    1 + F_1(p) + F_2(p) + \dots + F_k(p) = (1-p^{-(1+\alpha_1)})\dots(1-p^{-(1+\alpha_k)})
    \]
    and so $A_{\{\alpha_1,\dots,\alpha_k\}}(0) =1$.

    \item To obtain the mixed derivative that we are interested in, we Taylor expand each $\zeta (\frac{1}{2} + i \gamma + a_j)$ about $a_j=0$, take derivatives with respect to $a_j$ and then set them equal to zero.

    \item In the nomenclature of the `recipe' of Conrey, Farmer, Keating, Rubinstein and Snaith \cite{CFKRS05} and the Ratios Conjecture of Conrey, Farmer and Zirnbauer \cite{CFZ05}, the terms that appear in the integral in \eqref{eq:ZRatios} are only the zero- and one-swap terms. The reason for this will become clear in Section \ref{ProofConj3}.

    \item The error term in Conjecture~\ref{RatioMom} can be contentious. While some authors prefer an error term of $O\left(T^{1 - \delta}\right)$ for $0 < \delta < 1/2$, in this conjecture and our later examples, we suggest an error term of $O\left(T^{1/2 +\varepsilon}\right)$, in keeping with the original statements of the recipe/Ratios Conjecture.

    \item It is not obvious from inspection that the main term of \eqref{eq:ZRatios} is holomorphic in terms of the shift parameters. However the symmetries of the expression imply that the poles cancel to form a holomorphic function. Lemma 6.7 of \cite{CFZ07} exhibits an integral representation for the permutation sum that proves the holomorphy.
\end{enumerate}
\end{remark}

We will carry the assumptions in Conjecture \ref{RatioMom} throughout this paper, namely that we assume the Riemann Hypothesis and that the small shifts $\alpha_j$ satisfy $|\Re(\alpha_j)| < \frac{1}{4}$ and $|\Im(\alpha_j)| \ll_\varepsilon T^{1-\varepsilon}$ for every $\varepsilon > 0$, and that  $|\delta|<1/4$.

We will now place the conjecture within the setting of discrete moments and generalisations of the problem known as Shanks' Conjecture. We show how our viewpoint leads to new conjectures for higher derivatives and moments of the zeta function, which were previously out of reach.

Shanks~\cite{Shanks61} conjectured that $\zeta ' (\rho)$, where $\rho=\beta+ i\gamma$ ranges over the non-trivial zeros of the Riemann zeta function, is real and positive in the mean. He formed this conjecture by studying Haselgrove's tables of numerical values of the Riemann zeta function~\cite{Has60}. In looking for a straightforward proof that there are infinitely many simple zeros of $\zeta (s)$, Conrey, Ghosh and Gonek~\cite{CGG88}
were the first to prove that this conjecture is true. In particular, they showed that
\[
\sum_{0 < \gamma \leq T} \zeta ' (\rho) = \frac{1}{2} \frac{T}{2 \pi} \left( \log \frac{T}{2 \pi} \right)^2 + O(T\log T).
\]
Fuji gave lower order terms for this result in~\cite{Fuj94,Fujii12} with a power-saving error term under the Riemann Hypothesis.

Later, Kaptan, Karabulut and Y{\i}ld{\i}r{\i}m~\cite{KKY11} extended this result to all derivatives of $\zeta (s)$, showing that
\begin{equation*}
	\sum_{0 < \gamma \leq T} \zeta^{(n)} (\rho) =
	\frac{(-1)^{n+1}}{n+1} \frac{T}{2 \pi} \left( \log \frac{T}{2 \pi} \right)^{n+1} + O \bigl(T (\log T)^n \bigr).
\end{equation*}
The authors of this paper then found all lower order terms with a power-saving error term under the Riemann Hypothesis~\cite{HugPC22}. In particular, we noted that $\zeta ^{(n)} (\rho)$, where $\rho$ ranges over non-trivial zeros of the Riemann zeta function, is real and positive or negative in the mean depending on whether $n$ is odd or even, respectively.

A complete history of Shanks' Conjecture, up to (but excluding) the results of this paper and those in \cite{HugPC24}, can be found in~\cite{HMPC24}.

An obvious question is to ask what happens for higher moments of the derivatives of $\zeta (s)$. In a companion paper~\cite{HugPC24} we were able to conjecture the answer to this question to leading order for complex moments of the first derivative. This was done using two different methods: the first involved using a similar approach to that taken by Keating and Snaith \cite{KS00a} by considering the moments of the characteristic polynomial of random unitary matrices using Selberg's integral, and the second involved using a similar approach to that taken by Hughes, Gonek, and Keating \cite{GHK07} using the hybrid model of the zeta function.

\begin{conjecture}\label{conj:n1}
	Assume the Riemann Hypothesis. For $\Re(k)>-3$,
	\begin{equation*}
		\sum_{0 < \gamma \leq T} \zeta' \left( \frac{1}{2} + i \gamma \right)^{k} \sim \frac{1}{\Gamma(k +2)} \frac{T}{2 \pi} \left( \log \frac{T}{2\pi} \right)^{k+1} 
	\end{equation*}
    as $T \to \infty$, where $\Gamma (z)$ is the gamma function.
\end{conjecture}

\begin{remark}
	In \cite{HugPC24} we conjectured the leading order behaviour of the complex $k$\textsuperscript{th} moment for the first derivative, where $\Re(k)>-3$. However, in this paper the method restricts us to non-negative integer $k$.
\end{remark}

Using Conjecture~\ref{RatioMom}, we are able to generalise Conjecture \ref{conj:n1} to moments of mixed derivatives. Applying Conjecture~\ref{RatioMom} forces $k$ to be an integer. We will perform this calculation in Section \ref{LeadingOrder}.

\begin{conjecture}\label{MixMom}
	Assume the Riemann Hypothesis. For $n_1,\dots,n_k \in \NN$, the moments of mixed derivatives of $\zeta (s)$, evaluated at the non-trivial zeros of $\zeta (s)$, are given by
	\begin{multline*}
		\sum_{0 < \gamma \leq T} \zeta^{(n_1)} \left( \frac{1}{2} + i \gamma \right)\dots\zeta^{(n_k)} \left( \frac{1}{2} + i \gamma \right) \\
		\sim (-1)^{n_1+\dots+n_k +k}\frac{n_1!\dots n_k!}{(n_1+\dots+n_k +1)!}  \frac{T}{2 \pi} \left( \log \frac{T}{2\pi} \right)^{n_1+\dots+n_k+1}
	\end{multline*}
	as $T \rightarrow \infty$.
\end{conjecture}

Setting all the derivatives to be equal gives the obvious corollary.

\begin{conjecture}\label{kthMom}
	Assume the Riemann Hypothesis. For  a non-negative integer $k$ and $n \in \mathbb{N}$, the $k$\textsuperscript{th} moment of the $n$\textsuperscript{th} derivative of $\zeta (s)$, evaluated at the non-trivial zeros of $\zeta (s)$, is given by
	\begin{equation*}
		\sum_{0 < \gamma \leq T} \zeta^{(n)} \left(\frac{1}{2} + i \gamma  \right)^{k} \sim (-1)^{k(n +1)}\frac{(n!)^k}{(kn +1)!} \frac{T}{2 \pi} \left( \log \frac{T}{2\pi} \right)^{kn + 1}
	\end{equation*}
	as $T \rightarrow \infty$.
\end{conjecture}

While we have gained the moments of mixed derivatives to leading order in these conjectures, we have lost some of the information that we are after in this paper - namely the full asymptotic expansions for the moments of mixed derivatives. We can instead expand the functions in Conjecture~\ref{RatioMom} fully, to obtain full asymptotic expansions.

We give two examples of obtaining the full asymptotic using Conjecture~\ref{RatioMom}. The first example involves finding an integral form of the full asymptotic for the Generalised Shanks' Conjecture, first given in \cite{HugPC22}. The methodology employed in this paper gives that result in an equivalent (but simpler) integral form, and as has been seen in recent times, is the `correct' way of writing moments are as integrals.  It comes from taking $j=1$ and expanding about the shift $\alpha_1$.
\begin{theorem}\label{thm:Shanks}
	For $n$ a positive integer and $\rho=1/2 + i\gamma$ a non-trivial zero of the Riemann zeta function,
	\begin{multline*}
		\sum_{0 < \gamma \leq T} \zeta ^{(n)} \left(\frac{1}{2} + i \gamma \right) =  \frac{n!}{2 \pi} \int_{1}^{T} \left( A_n + \frac{(-1)^{n+1} L^{n+1}}{(n+1)!} + \sum_{m=0}^{n} \frac{(-1)^{m+1} L^m \gamma_{n-m}}{m! (n-m)!} \right) \ dt + O \left( T^{1/2 +\varepsilon} \right),
	\end{multline*}
	where $L = \log \frac{t}{2\pi}$, and the $\gamma_n$ are the coefficients in the Laurent expansion of $\zeta (s)$ about $s=1$, given by
    \begin{equation}\label{eq:zetalaur}
    \zeta(s) = \frac{1}{s-1} +\gamma_0 - \gamma_1 (s-1) +\dots + \frac{(-1)^n}{n!}\gamma_n (s-1)^n +O\left((s-1)^{n+1}\right)
    \end{equation}
    and the $A_n$ are the coefficients in the Laurent expansion of $\zeta' (s)/\zeta (s)$ about $s=1$, given by
    \begin{equation*}
    \frac{\zeta ' (s)}{\zeta (s)} = - \frac{1}{s-1} +A_0 + A_1 (s-1) + \dots + A_n (s-1)^n +O\left((s-1)^{n+1}\right),
    \end{equation*}
\end{theorem}
Note that while we have assumed the Riemann Hypothesis in this result, in \cite{HugPC22} it is actually proven unconditionally. The only change occurs in the error term when we assume the Riemann Hypothesis, and in that paper, we show that it truely is $O \left( T^{1/2 +\varepsilon} \right)$ (in fact, we make the $T^\varepsilon$ explicit in terms of powers of a logarithm).

The second example is a conjecture for the second moment of the first derivative of the Riemann zeta function, a previously unknown result. Specifically, we conjecture the following result.
\begin{conjecture}\label{conj:secmom}
	For $L = \log \frac{t}{2\pi}$ and $\rho=1/2 + i\gamma$ a non-trivial zero of the Riemann zeta function,
	\begin{multline*}
		\sum_{0 < \gamma \leq T} \zeta ' \left(\frac{1}{2} + i \gamma \right)^2 = \\
		\frac{1}{2 \pi} \int_{1}^{T} \left( \frac{1}{6}L^3 + \frac{1}{2}L^2 \left(2\gamma_0 +A^{(0,0,1)}  \right) \right. + \frac{1}{2} L \left( -8\gamma_1 + 4\gamma_0 A^{(0,0,1)} + A^{(0,0,2)}  + 2A^{(0,1,1)} \right) \\
		+ \frac{1}{6} \left( -12\gamma_0^3 - 36\gamma_0 \gamma_1 + 6\gamma_2 -24 \gamma_1 A^{(0,0,1)} + 6 \gamma_0 A^{(0,0,2)} + A^{(0,0,3)} \right. \\
		\left. + 12 \gamma_0 A^{(0,1,1)} + 3 A^{(0,1,2)}  -3 A^{(0,2,1)} +6 A^{(1,1,1)}  \right) \ dt + O\left(T^{1/2 +\varepsilon}\right)
	\end{multline*}
	where the $\gamma_n$ are the coefficients in the Laurent expansion of $\zeta (s)$ about $s=1$ given in \eqref{eq:zetalaur} and the $A^{(i,j,k)}$ are arithmetic terms that are various products and sums over primes.
\end{conjecture}

We finish this results section by reminding the reader that if we want to obtain the full asymptotic for any moment of mixed derivatives, then all we need to do is expand the terms in the integral in Conjecture~\ref{RatioMom}. A computer package is helpful to perform these expansions, and to calculate numerically what the arithmetic terms should equal.

\section{Outline of the Paper}
As we have seen in the introduction, Conjecture~\ref{RatioMom} and its consequences lies at the intersection of several well-established areas, each with a rich historical background.

In Section~\ref{sect:MomRMTHis}, we place our results within the framework of random matrix theory. We explain the relevance of random matrix models to number-theoretic moments and present two new results concerning shifted moments and moments of mixed derivatives of characteristic polynomials, evaluated at matrix eigenvalues. Theorem~\ref{thm:CharPolyDeriv} is the random matrix analogue of Conjecture~\ref{MixMom}.

Section~\ref{sect:RatHis} discusses the origins of the Ratios Conjecture and outlines the methodology used to derive conjectures using this framework. This methodology forms the basis for our later formulation of Conjecture~\ref{RatioMom}.

Section~\ref{Prelim} presents a set of preliminary lemmas, particularly involving relations between Vandermonde determinants and links to complete homogeneous symmetric polynomial. These lemmas play a key role in evaluating various functions involving the shifts introduced in our framework.

In Section~\ref{sect:RMTMot}, we prove the random matrix results stated earlier. The calculations involve solving a recurrence relation for certain Toeplitz determinants, which arise naturally when computing expectations over the unitary group. These results rely on the lemmas introduced in Section~\ref{Prelim}.

Section~\ref{ProofConj3} applies the Ratios Conjecture framework to formulate Conjecture~\ref{RatioMom}. This involves expressing the shifted zeta moment as a Cauchy integral, and applying standard tricks to evaluate the integral. We show that many contributions vanish into the error term, and the remaining integral is amenable to analysis via the Ratios methodology from Section~\ref{sect:RatHis}.

In Section~\ref{LeadingOrder}, we derive the leading-order behaviour of the moments of mixed derivatives, as stated in Conjecture~\ref{MixMom}. This requires expanding the integrand in Conjecture~\ref{RatioMom} and carefully analyzing its dependence on the shifts.

Finally, Section~\ref{Examples} illustrates the power of Conjecture~\ref{RatioMom} through two worked examples. In Example~\ref{ShanksInt}, we confirm that the Ratios Conjecture predicts the correct asymptotic for the first discrete moment of the derivatives of $\zeta(s)$, reproducing the integral form of the result from~\cite{HugPC22}. In Example~\ref{2ndMom}, we present, for the first time, an explicit conjecture for the second discrete moment of $\zeta'(s)$, supported by numerical evidence.

\section{Moments of the zeta function through random matrix theory}\label{sect:MomRMTHis}

In this section we describe the philosophy of Keating and Snaith \cite{KS00a} where random matrix theory was used to conjecture moments of the zeta function.

It had been conjectured that the limiting distribution of the non-trivial zeros of $\zeta (s)$, on the scale of their mean spacing, is asymptotically the same as that of the eigenvalues of matrices chosen according to Haar measure, in the limit as the matrix size tends to infinity.

The function that the zeros of the zeta function are the zeros of is, by definition, the zeta function. Keating and Snaith asked which function are the eigenvalues the zeros of? This is exactly the characteristic polynomial, so in \cite{KS00a} Keating and Snaith proposed that the characteristic polynomial of a large random unitary matrix can be used to model the value distribution of the Riemann zeta function near a large height $T$.

They set
\begin{align}
Z(\theta) &= \det\left(I-U e^{i\theta}\right) \notag\\
&=\prod_{m=1}^N \left(1-e^{-i\theta_m} e^{i\theta} \right) \label{eq:CharPoly}
\end{align}
for the characteristic polynomial of a random unitary matrix $U$, where $e^{i\theta_1},\dots,e^{i\theta_N}$ are the eigenvalues the matrix $U$.

They calculated the moments of the characteristic polynomial of random unitary matrices, averaged over Haar measure on $N\times N$ unitary matrices, and proved the following exact result.
\begin{theorem}[Keating--Snaith]
Let $\E_N$ denote the expectation over Haar measure on the $N\times N$ unitary matrices, and let $Z(\theta)$ be the characteristic polynomial given by \eqref{eq:CharPoly}. For any fixed real $\theta$ and $\Re(k)>~-1/2$,
\begin{align*}
\E_N\left[ \left|Z(\theta)\right|^{2k} \right] &= \prod_{j=1}^N \frac{\Gamma(j) \Gamma(j+2k)}{\Gamma(j+k)^2}\\
&=  \frac{G^2(k+1) G(2k+N+1) G(N+1)}{G(2k+1) G^2(k+N+1)}
\end{align*}
where $G(z)$ is the Barnes $G$-function. For $k$ fixed and $N\to\infty$ this satisfies
\[
\E_N\left[ \left|Z(\theta)\right|^{2k} \right] \sim \frac{G^2(k+1)}{G(2k+1)} N^{k^2}.
\]
\end{theorem}

They then argued that these results should (in some precise way) match the equivalent moments of the zeta function, at least at leading order. By comparing their result with all known and conjectured results for the moments of zeta, they conjectured the following result.
\begin{conjecture}[Keating--Snaith]
For any fixed $k$ with $\Re (k) > -1/2$,
\begin{equation*}
     \int_{0}^{T} \left|\zeta \left(\frac{1}{2} + it \right) \right|^{2k} \ dt \sim  \frac{G^2 (k+1)}{G(2k+1)} a(k) T \left( \log \frac{T}{2\pi} \right)^{k^2}
\end{equation*}
as $T \rightarrow \infty$, where $G(z)$ is the Barnes $G$-function and
\begin{equation}\label{AF}
    a(k) = \prod_{p \text{ prime}} \left( 1 - \frac{1}{p} \right)^{k^2} \sum_{m=0}^{\infty} \left( \frac{\Gamma (m+k)}{m! \, \Gamma (k)} \right)^2 p^{-m}
\end{equation}
is an arithmetic factor.
\end{conjecture}

To motivate our conjectural results for the zeta function, we prove the following result for the characteristic polynomial. First we need to define $h_m(\alpha_1,\dots,\alpha_k)$ to be the complete homogeneous symmetric polynomial of degree $m$ in $k$ variables, namely
\begin{equation}\label{eq:homoSympoly}
h_m(\alpha_1,\dots,\alpha_k) = \sum_{\substack{n_1 \geq 0,\dots,n_k \geq 0 \\ n_1 + \dots + n_k = m}} \alpha_1^{n_1} \alpha_2^{n_2} \dots \alpha_k^{n_{k}}.
\end{equation}

\begin{theorem}\label{thm:CharPoly}
Let $Z_U(\theta)$, given by \eqref{eq:CharPoly}, be the characteristic polynomial of a unitary matrix with eigenvalues $e^{i\theta_1},\dots,e^{i\theta_N}$. Then
\begin{multline*}
\E_N\left[\frac{1}{N} \sum_{n=1}^N Z(\theta_n+\alpha_1) \dots Z(\theta_n+\alpha_k)\right] = \frac{1}{N}\prod_{r=1}^k \left(1-e^{-i\alpha_r}\right) h_{N-1}\left(e^{-i\alpha_1} , \dots, e^{-i\alpha_k} , 1,1 \right)\\
= \frac{1}{N} \left(\sum_{r=1}^{k} \frac{e^{-i(N+k)\alpha_r} }{1-e^{-i\alpha_r}} \prod_{\substack{j=1 \\ j \neq r}}^{k} \frac{1-e^{-i\alpha_j}}{(e^{-i\alpha_r} - e^{-i\alpha_j})} + N+k - \sum_{\ell=1}^k \frac{1}{1-e^{-i\alpha_\ell}} \right)
\end{multline*}
where $h_{N-1}\left(e^{-i\alpha_1} , \dots, e^{-i\alpha_k} , 1,1 \right)$ is the complete homogenous symmetric polynomial of degree $N-1$ in $k+2$ variables.
\end{theorem}

The second exact expression is amenable to asymptotic analysis when $N$ is large, although it involves some combinatorial analysis to show the singularities cancel. Furthermore, one can Taylor expand to find expressions for the derivative of the characteristic polynomial, evaluated at an eigenvalue. Such expressions include all the terms, however, one can extract a particularly simple expression for the leading order behaviour of the discrete moments of the derivative.

Specifically, in Section~\ref{sect:RMTMot} we will prove the following two results.
\begin{theorem}\label{thm:CharPolyShiftsLeadingOrder}
As $N \to \infty$ uniformly for bounded $a_1,\dots,a_k$, we have
\[
\lim_{N \to \infty} \E_N\left[ \prod_{r=1}^k Z\left(\theta_N+\frac{a_r}{N}\right) \right]
=   \sum_{m=0}^\infty \frac{(-1)^{m} i^{k+m}}{(m+k+1)!} h_{m}(a_1,\dots,a_k)   \prod_{j=1}^k a_j
\]
where $h_m(\alpha_1,\dots,\alpha_k)$ is given by \eqref{eq:homoSympoly}.
\end{theorem}

\begin{theorem}\label{thm:CharPolyDeriv}
    Let $Z_U(\theta)$ be given by \eqref{eq:CharPoly}, and the $e^{i\theta_n}$ be the eigenvalues of the matrix $U$. For $n_1,\dots,n_k \in \mathbb{N}$, we have
    \begin{equation*}
    \E_N\left[ \frac{1}{N} \sum_{n=1}^N Z^{(n_1)}(\theta_n)\dots Z^{(n_k)}(\theta_n)\right]
    \sim (-1)^{n_1+\dots+n_k-k} i^{n_1+\dots+n_k} \frac{n_1! \dots n_k!}{(n_1+\dots+n_k+1)!} N^{n_1+\dots+n_k}
    \end{equation*}
    as $N\to\infty$, where $Z^{(n)}(\theta)$ denotes the $n$\textsuperscript{th} derivative of $Z(\theta)$.
\end{theorem}

\section{The Ratios Conjecture}\label{sect:RatHis}

In keeping with the theme of this paper, we could ask if it is possible to extend the ideas of Keating and Snaith \cite{KS00a} to obtain conjectures for the full asymptotic of the moments of the zeta function. This problem was resolved conjecturally for integer moments in the `recipe' of Conrey, Farmer, Keating, Rubinstein and Snaith \cite{CFKRS05}.
	\begin{conjecture}[Conrey--Farmer--Keating--Rubinstein--Snaith]
		For $k$ a fixed integer,
		\begin{equation*}
			I_{2k} (T) = \int_{1}^{T} \left| \zeta \left( \frac{1}{2} + it \right) \right|^{2k} \ dt = \int_{1}^{T} \mathcal{P}_k \left(\log \frac{t}{2\pi} \right) \ dt + O\left(T^{1/2 + \varepsilon} \right)
		\end{equation*}
		as $T \rightarrow \infty$, where $\mathcal{P}_k (x)$ is a calculable polynomial of degree $k^2$ in $x$.
	\end{conjecture}
	
The recipe was created through comparison with the analogous quantities for the characteristic polynomials of matrices averaged over classical compact groups \cite{CFKRS03}.

Rather than asking if we can calculate a product of zeta functions as in the recipe, we could also ask whether we can calculate a moment of a quotient of products of zeta functions. Farmer~\cite{Far93} conjectured that we can do exactly this in the case of two zeta functions on the numerator and two zeta functions on the denominator.
\begin{conjecture}[Farmer]
For complex shifts $\alpha, \beta, \gamma, \delta \asymp 1/\log T$ with small positive real parts and with $s=1/2+it$,
\[
\frac{1}{T} \int_{1}^{T} \frac{\zeta ( s + \alpha) \zeta ( s + \beta)}{\zeta ( s + \gamma) \zeta ( s + \delta)} \ dt \sim \frac{(\alpha + \delta)(\beta + \gamma)}{(\alpha + \beta)(\gamma + \delta)} - T^{-(\alpha + \beta)} \frac{(\alpha - \gamma)(\beta - \delta)}{(\alpha + \beta)(\gamma + \delta)}
\]
as $T \rightarrow \infty$.
\end{conjecture}
The Ratios Conjecture \cite{CFZ05} generalises this to quotients of an arbitrary number of $L$-functions averaged over a family, with a full asymptotic expansion and a power-saving error term in each case. 

The Ratios Conjecture was developed through analogy with quotients of products of characteristic polynomials \cite{CFZ07}, exploiting the connection between moments in random matrix theory and moments of the zeta function.

For an integral containing quotients of products of zeta functions (or more generally $L$-functions), we perform the follow steps (ignoring all error terms throughout):
\begin{enumerate}
    \item Replace all zeta functions in the denominator by its Dirichlet series (which we know conditionally converges up to the critical line under the Riemann Hypothesis).
    \item Replace all zeta functions in the numerator by the approximate functional equation.
    \item Expand the new terms in the integral, keeping only those terms with an equal number of $\chi (s)$ and $\chi (1-s)$ terms. These are the terms that are not rapidly oscillating.
    \item In these remaining terms, keep only the diagonal terms.
    \item Complete the resulting sums, factoring out various zeta factors to make the resulting terms converge.
    \item A power-saving error term is then conjectured for the final result.
\end{enumerate}
A full example will be seen when we use the Ratios Conjecture to form Conjecture~\ref{RatioMom} in Section \ref{ProofConj3}.

Note that this methodology is the same as the recipe methodology, with the first step of replacing zeta functions in the denominator by their Dirichlet series being the only additional step.

The Ratios Conjecture can be used for many computations, for example as shown in Conrey and Snaith~\cite{ConSna07}. They showed how to calculate $n$-level correlations of zeros with lower order terms, averages of mollified $L$-functions, non-vanishing results for various families of $L$-functions, and most importantly for the analogous results in this paper, discrete moments of $|\zeta (s)|$. The strength of the Ratios Conjecture is that it can reproduce many known results in a simple way, as well as predicting results that are out of our reach of proving rigorously.

\section{Preliminary Lemmas on Vandermonde Determinants}\label{Prelim}

As part of our later calculations, a certain identity between Vandemonde determinants will prove useful.

Let
\begin{align*}
\Delta_k\left(\alpha_1,\dots,\alpha_k\right)
&\renewcommand{\arraystretch}{1.4} = \det \begin{pmatrix}
1 & \alpha_1 & \alpha_1^2 & \dots & \alpha_1^{k-1}  \\
1 & \alpha_2 & \alpha_2^2 & \dots & \alpha_2^{k-1}  \\
\vdots & \vdots & \vdots & \dots & \vdots \\
1 & \alpha_{k} & \alpha_{k}^2 & \dots & \alpha_{k}^{k-1}
\end{pmatrix} \\
&= \prod_{1 \leq i < j \leq k} \left(\alpha_j - \alpha_i \right)
\end{align*}
denote the Vandermonde determinant with $k$ parameters, and also let
\[
\Delta_{k-1}\left(\alpha_1,\dots,\alpha_k ; \alpha_\ell\right) = \prod_{\substack{1 \leq i < j \leq k \\ i \neq \ell \\ j \neq \ell}} \left(\alpha_j - \alpha_i \right)
\]
denote the Vandermonde determinant with the $\ell$\textsuperscript{th} element removed.

Recall from \eqref{eq:homoSympoly} that $h_m(\alpha_1,\dots,\alpha_k)$ denotes the complete homogeneous symmetric polynomial of degree $m$ in $k$ variables, namely
\[
h_m(\alpha_1,\dots,\alpha_k) = \sum_{\substack{n_1 \geq 0,\dots,n_k \geq 0 \\ n_1 + \dots + n_k = m}} \alpha_1^{n_1} \alpha_2^{n_2} \dots \alpha_k^{n_{k}}.
\]

\begin{lemma}\label{lem:vandermonde}
For $k\in \mathbb{N}$ and $n\in \mathbb{Z}$, if $\alpha_1,\dots,\alpha_k \neq 0$ then
\begin{multline*}
\sum_{\ell=1}^k (-1)^{k+\ell}   \alpha_\ell^{n-1}  \Delta_{k-1}\left(\alpha_1,\dots,\alpha_k ; \alpha_\ell\right)\\
=
\begin{cases}
\displaystyle \Delta_{k}\left(\alpha_1,\dots,\alpha_{k}\right) h_{n-k}\left(\alpha_1, \dots \alpha_k\right) & \text{ if } n \geq k \\
0 & \text{ if } 1\leq n \leq k-1 \\
\noalign{\vskip1ex}
\displaystyle (-1)^{k+1} \Delta_{k}\left(\alpha_1,\dots,\alpha_{k}\right)   \frac{ h_{|n|}\left(\alpha_1^{-1}, \dots \alpha_k^{-1}\right)}{\alpha_1 \alpha_2 \dots \alpha_k} & \text{ if } n \leq 0.
\end{cases}
\end{multline*}
\end{lemma}

\begin{proof}
Let
\begin{equation*}
X = \sum_{\ell=1}^{k} (-1)^{k+\ell} \alpha_\ell^{n-1} \Delta_{k-1}\left(\alpha_1,\dots,\alpha_k ; \alpha_\ell\right) .
\end{equation*}

Observe that $X$ can be expressed as a determinant,
\[
\renewcommand{\arraystretch}{1.4}
X = \det \begin{pmatrix}
1 & \alpha_1 & \alpha_1^2 & \dots & \alpha_1^{k-2} & \alpha_1^{n-1}  \\
1 & \alpha_2 & \alpha_2^2 & \dots & \alpha_2^{k-2} & \alpha_2^{n-1}  \\
\vdots & \vdots & \vdots & \dots & \vdots & \vdots\\
1 & \alpha_{k} & \alpha_{k}^2 & \dots & \alpha_{k}^{k-2}  & \alpha_k^{n-1}
\end{pmatrix}
\]
where the terms in the last column are replaced by $\alpha_j^{n-1}$ (in the standard Vandermonde these terms would be $\alpha_j^{k-1}$). Expanding the determinant down its last column, it is clear that it equals $X$.

Note that if $n\leq k-1$ then the last column will be equal to one of the other columns in the matrix, and thus the determinant will be zero.

To complete the proof of the lemma, we now evaluate $X$ in a  different manner, via row manipulations.

Subtract the first row from all subsequent rows. Pull out a factor of $(\alpha_j-\alpha_1)$ from the $j$\textsuperscript{th} row (for $j\geq 2$).

Then subtract the new second row from all subsequent rows, and pull out a factor of $(\alpha_j-\alpha_2)$ from the $j$\textsuperscript{th} row (for $j\geq 3$ this time).

Repeat this process all the way down to subtracting what remains in the the $k-1$\textsuperscript{st} row from what remains in the $k$\textsuperscript{th} row, and pulling out a factor of $(\alpha_k-\alpha_{k-1})$.

The terms that are factored out during this process are
\[
\prod_{\ell=1}^{k-1} \prod_{j=\ell+1}^k \left(\alpha_j-\alpha_\ell\right)  = \Delta_{k}\left(\alpha_1,\dots,\alpha_{k}\right)
\]
and the remaining matrix determinant is
\[
\det
\begin{pmatrix}
A_{1,1} & A_{1,2} & \dots & A_{1,k-1} & B_{1} \\
A_{2,1} & A_{2,2} & \dots & A_{2,k-1} & B_{2} \\
\vdots & \vdots & \vdots & \vdots & \vdots\\
A_{k,1} & A_{k,2} & \dots & A_{k,k-1} & B_{k} \\
\end{pmatrix}
\]
where
\[
A_{\ell,m} = \sum_{\substack{n_1,\dots,n_{\ell} \geq 0 \\ n_1 + \dots +n_\ell  = m-\ell}} \alpha_1^{n_1} \dots \alpha_{\ell}^{n_\ell}
\]
and where $B_\ell = A_{\ell,n}$ in the case $n \geq 1$, and in the case $n \leq 0$,
\[
B_{\ell} = (-1)^{\ell+1} \sum_{\substack{n_1,\dots,n_{\ell} \geq 0 \\ n_1 + \dots +n_\ell = |n|}} \alpha_1^{-n_1-1} \dots \alpha_{\ell}^{-n_\ell-1}.
\]

Crucially note that $A_{\ell,m}=0$ for $m<\ell$ and $A_{\ell,\ell}=1$. This makes the determinant trivial to evaluate - the matrix is upper triangular, and so the determinant will equal $B_k$.

That is, we have shown that
\[
X = \Delta_{k}\left(\alpha_1,\dots,\alpha_{k}\right) B_k
\]
and this is the right-hand of the statement of the lemma, since if $1 \leq n\leq k-1$ we have
\[
B_k = A_{k,n} = 0
\]
and if $n \geq k$ we have
\begin{align*}
B_k &= A_{k,n} = \sum_{\substack{n_1,\dots,n_{k} \geq 0 \\ n_1 + \dots +n_k  = n-k}} \alpha_1^{n_1} \dots \alpha_{k}^{n_k} \\
&= h_{n-k}\left(\alpha_1,\dots,\alpha_k\right)
\end{align*}
and if $n \leq 0$ we have
\begin{align*}
B_k &= (-1)^{k+1} \sum_{\substack{n_1,\dots,n_k \geq 0 \\ n_1 + \dots +n_k = |n|}} \alpha_1^{-n_1-1} \dots \alpha_{k}^{-n_k-1} \\
&= \frac{ h_{|n|}\left(\alpha_1^{-1}, \dots \alpha_k^{-1}\right)}{\alpha_1 \alpha_2 \dots \alpha_k}
\end{align*}
as required.
\end{proof}

If we assume $\alpha_1,\dots,\alpha_k \in \mathbb{C}$ are all distinct (so $\Delta(\alpha_1,\dots,\alpha_k)$ is non-zero) we can adapt this lemma into some useful combinatorial sums we will need in our subsequent arguments.

\begin{lemma}\label{LemmaProdSum}
For $k \in \mathbb{N}$, $n \in \mathbb{Z}$, and distinct nonzero $\alpha_1,\dots,\alpha_k \in \mathbb{C}$, we have
\begin{equation*}
\sum_{\ell=1}^{k} \alpha_\ell^n \prod_{ \substack{j=1 \\ j \neq \ell}}^{k} \frac{\alpha_j}{(\alpha_j - \alpha_\ell)} =
\begin{cases}
\displaystyle (-1)^{k+1} h_{n-k}\left(\alpha_1,\dots,\alpha_k\right) \prod_{j=1}^k \alpha_j & \text{ if } n \geq k\\
0 & \text{ if } 1 \leq n \leq k-1 \\
h_{|n|}\left(\alpha_1^{-1},\dots,\alpha_k^{-1}\right) & \text{ if } n \leq 0
\end{cases}
\end{equation*}
where $h_m(\alpha_1,\dots,\alpha_k)$ is the degree-$m$ homogenous symmetric polynomial in $k$ variables, given in \eqref{eq:homoSympoly}.
\end{lemma}

\begin{proof}
If we assume $\alpha_1,\dots,\alpha_k \in \mathbb{C}$ are all distinct, then we can divide both sides in Lemma~\ref{lem:vandermonde} by $\Delta_k(\alpha_1,\dots,\alpha_k)$ since it is non-zero when all its arguments are distinct. We have
\begin{equation*}
\frac{\Delta_{k-1}(\alpha_1,\dots,\alpha_k ; \alpha_\ell)}{\Delta_k(\alpha_1,\dots,\alpha_k)} = (-1)^{\ell-1} \prod_{\substack{j=1\\j \neq \ell}}^k \frac{1}{\alpha_j-\alpha_\ell}
\end{equation*}
where we observe that for $j=1,\dots,\ell-1$ there are terms  $\alpha_\ell-\alpha_j$ in the Vandermonde, whereas in the product on the RHS they are  $\alpha_j-\alpha_\ell$, thus accounting for the $(-1)^{\ell-1}$ term.

Therefore Lemma~\ref{lem:vandermonde} directly translates to
\begin{equation}\label{eq:combsum}
\sum_{\ell=1}^k (-1)^{k-1}   \alpha_\ell^{n-1}   \prod_{\substack{j=1\\j \neq \ell}}^k \frac{1}{\alpha_j-\alpha_\ell}
=
\begin{cases}
\displaystyle  h_{n-k}\left(\alpha_1, \dots \alpha_k\right) & \text{ if } n \geq k \\
0 & \text{ if } 1\leq n \leq k-1 \\
\noalign{\vskip1ex}
\displaystyle (-1)^{k+1} \frac{ h_{|n|}\left(\alpha_1^{-1}, \dots \alpha_k^{-1}\right)}{\alpha_1 \alpha_2 \dots \alpha_k} & \text{ if } n \leq 0.
\end{cases}
\end{equation}
Multiplying both sides by $\prod_{j=1}^k \alpha_j$ gives the lemma as stated.
\end{proof}

We now restate various special cases of this lemma in such a way that will be directly applicable in later sections.

Setting $n=-1$ in Lemma~\ref{LemmaProdSum} yields
\begin{equation}\label{eq:LemmaZero}
\sum_{\ell=1}^{k} \frac{1}{\alpha_\ell} \prod_{ \substack{j=1 \\ j \neq \ell}}^{k} \frac{\alpha_j}{(\alpha_j - \alpha_\ell)} = \sum_{\ell=1}^{k} \frac{1}{\alpha_\ell}
\end{equation}
and setting $n=0$ in Lemma~\ref{LemmaProdSum} yields
\begin{equation}\label{eq:lem:n0Term}
\sum_{\ell=1}^k \prod_{\substack{j=1\\j\neq \ell}}^k \frac{\alpha_j}{\alpha_j-\alpha_\ell} = 1.
\end{equation}

\section{Proofs of the random matrix theorems}\label{sect:RMTMot}

In this section we will prove certain results on shifted moments of characteristic polynomials. The motivation for using random matrix models for the zeta function is explained more fully in our companion paper, \cite{HugPC24}, Additionally, in this case the combinatorial identities required are the same as needed in the more complicated ratios methodology, so this is also a fruitful exercise to help understand the combinatorics.

\subsection{Expressing the expectation as a Toeplitz determinant} \hfill

Let
\[
Z(\theta) = \prod_{m=1}^N \left(1-e^{i \theta_m} e^{-i\theta}\right)
\]
be the characteristic polynomial of an $N \times N$ unitary matrix and consider the Haar measure over such matrices, which has Weyl density
\begin{equation}\label{eq:Weyl}
d\mu_N = \frac{1}{N! (2 \pi)^N} \prod_{1 \leq j < \ell \leq N} |e^{i\theta_j} - e^{i\theta_\ell}|^2 \ d\theta_1 \ldots d\theta_N.
\end{equation}
By rotation invariance of Haar measure, there is no preferred eigenvalue, so
\begin{equation*}
\E_N\left[\frac{1}{N} \sum_{n=1}^N Z(\theta_n+\alpha_1) \dots Z(\theta_n+\alpha_k)\right] = \E_N\left[ Z(\theta_N+\alpha_1) \dots Z(\theta_N+\alpha_k)\right]
\end{equation*}
and we therefore need to evaluate
\begin{multline*}
    \E_N\left[ Z(\theta_N+\alpha_1) \dots Z(\theta_N+\alpha_k)\right] \\
    = \frac{1}{N! (2 \pi)^N} \idotsint_{-\pi}^\pi \prod_{1 \leq j < \ell \leq N} |e^{i\theta_j} - e^{i\theta_\ell}|^2 \prod_{r=1}^k \prod_{m=1}^N \left(1-e^{i \theta_m} e^{-i\theta_N} e^{-i\alpha_r}\right)  d\theta_m.
\end{multline*}
If we change variables to $\phi_m = \theta_m-\theta_N$ (for $m=1,\dots,N-1$), make use of the $2\pi$-periodicity, and note that $\theta_N$ integral is now trivial, we see that after factoring the Weyl density and some further manipulation, this equals
\begin{multline*}
\prod_{r=1}^k \left(1-e^{-i\alpha_r}\right) \frac{1}{N \times (N-1)! (2 \pi)^{N-1}} \idotsint_{-\pi}^\pi \prod_{1 \leq j < \ell \leq N-1} |e^{i\phi_j} - e^{i\phi_\ell}|^2 \times \prod_{j=1}^{N-1} |e^{i\phi_m} - 1|^2 \\
\times\prod_{m=1}^{N-1} \prod_{r=1}^k\left(1-e^{i \phi_m} e^{-i\alpha_r} \right) d\phi_m.
\end{multline*}
Note that this can be reinterpreted in terms of an expectation over Haar measure over $(N-1)\times(N-1)$ unitary matrices, and consequently (using Heine's identity) as a Toeplitz determinant. That is, we have shown
\begin{equation}
\E_N\left[ \prod_{r=1}^k Z(\theta_N+\alpha_r) \right] = \frac{1}{N}\prod_{r=1}^k \left(1-e^{-i\alpha_r}\right) D_{N-1}[f] \label{eq:ShiftedCharPolyInTermsOfDN}
\end{equation}
where the Toeplitz determinant
\[
D_{N-1}[f] = \det  \begin{pmatrix}
	\widehat f_0 & \widehat f_0 &  \widehat f_1 &  \dots & \widehat f_{N-1} \\
	\widehat f_{-1} & \widehat f_0 &  \widehat f_1 &  \dots & \widehat f_{N-2} \\
	\widehat f_{-2} & \widehat f_{-1} &  \widehat f_0 &  \dots & \widehat f_{N-3} \\
	\ddots & \ddots & \ddots & \dots \\
	\widehat f_{-(N-1)} & \widehat f_{-(N-2})  & \widehat f_{-(N-3)} & \dots &  \widehat f_0
\end{pmatrix}
\]
has the symbol
\begin{align}
	f(\phi) &=\left|e^{i\phi}-1\right|^2   \left(1-e^{i\phi} e^{-i\alpha_1} \right) \dots  \left(1-e^{i\phi} e^{-i\alpha_k} \right) \notag \\
&=-e^{-i\phi} \left(1-e^{i\phi}\right)^2   \left(1-e^{i\phi} e^{-i\alpha_1} \right) \dots  \left(1-e^{i\phi} e^{-i\alpha_k} \right). \label{eq:f}
\end{align}

\subsection{Expressing the expectation as a linear recurrence relation} \hfill

It will be slightly simpler to evaluate $D_{N}[f]$
with symbol
\begin{equation}\label{eq:fwithA}
f(\phi) = -e^{-i\phi} \prod_{r=1}^{k+2} \left(1-e^{i\phi} A_r \right)
\end{equation}
with distinct $A_r$, where eventually we will change $N \mapsto N-1$, and set $A_{r} = e^{-i \alpha_r}$ for $r=1,\dots,k$, and take the limit as $A_{k+1} \to A_{k+2} =1$. This will allow us to recover \eqref{eq:ShiftedCharPolyInTermsOfDN} via \eqref{eq:f}.

\begin{lemma}\label{lem:fHat}
The Fourier coefficients of $f(\phi)$ given in \eqref{eq:fwithA} are
\begin{align*}
    \widehat f_{-1} &= -1 \\
    \widehat f_\ell &= (-1)^\ell C_{\ell+1} && \text{for $0 \leq \ell \leq k+1$} \\
    \widehat f_{\ell} &= 0 && \text{for $\ell \geq k+2$ or for $\ell \leq -2$}
\end{align*}
where
\[
C_j = C_j(A_1,\dots,A_{k+_2}) = \sum_{\substack{I \subseteq \{1,\dots,k+2\} \\ |I|=j}} \prod_{i \in I} A_i
\]
is the $j$\textsuperscript{th} elementary symmetric polynomial in $k+2$ variables.
\end{lemma}

\begin{proof}
First note that
\[
\prod_{r=1}^{k+2} \left(1-X A_r\right)  = 1 - C_1 X +  C_2 X^2 -  C_3 X^3 + \dots + (-1)^{k+2}  C_{k+2} X^{k+2}
\]
with the $C_j$ given in the lemma.

Next multiply this by $-1/X$, and pick out the coefficient of $X^\ell$ (whose exponent can now range from $-1$ to $k+1$) to see the Fourier coefficients are
\begin{align*}
\widehat f_{-1}&=-1\\
\widehat f_0 &=  C_1\\
\widehat f_1 &= - C_2\\
&\vdots\\
\widehat f_{k+1} &= (-1)^{k+1} C_{k+2}
\end{align*}

This yields the values of $\widehat f_{\ell}$ as given in the lemma.
\end{proof}

\begin{lemma}\label{lem:recurrence}
    For the symbol $f$ given in \eqref{eq:fwithA}, for $N \geq k+2$ the Toeplitz determinant satisfies a finite linear recurrence relation
    \begin{equation*}
    D_{N}[f] = \sum_{j=1}^{k+2} (-1)^{j+1} C_j D_{N-j}[f]
    \end{equation*}
    with the $k+2$ initial conditions given by
    \begin{align*}
        D_0[f] &= 1\\
        D_1[f] &= C_1\\
        D_2[f] &= C_1 D_1[f] - C_2 \\
        D_3[f] &= C_1 D_2[f] - C_2 D_1[f] + C_3\\
        &\vdots\\
        D_{k+1}[f] &= \sum_{j=1}^{k+1} (-1)^{j+1} C_j D_{k+1-j}[f],
    \end{align*}
    where $C_j = C_j(A_1,\dots,A_{k+_2})$ is the $j$\textsuperscript{th} elementary symmetric polynomial in $k+2$ variables given in Lemma~\ref{lem:fHat}.
\end{lemma}

\begin{proof}
    Recall the $D_N[f]$ is the Toeplitz matrix with $(i,j)$\textsuperscript{th} entry equal to $\widehat f_{j-i}$. Since the previous Lemma showed that $\widehat f_{-1} = -1$ and $\widehat f_\ell = 0$ for $\ell \leq -2$ we see that the matrix in question is upper triangular, apart from $-1$ on the lower diagonal. This type of determinant is simple to evaluate by repeatedly expanding down the first column, and equals
 \[
D_{N}[f] = \widehat f_0 D_{N-1}[f] + \widehat f_1 D_{N-2}[f] + \dots + \widehat f_{N-2} D_1[f] + \widehat f_{N-1}
\]

Next, use Lemma~\ref{lem:fHat} again to replace $\widehat f_\ell$ with their values in terms of $A_1, \dots, A_{k+2}$. Finally,  noticing that $\widehat f_{\ell}=0$ for $\ell \geq k+2$ makes the expression a finite linear recurrence relation.
\end{proof}

\begin{lemma}\label{lem:RecurrenceGeneralSolution}
The recurrence given in Lemma~\ref{lem:recurrence} has a general solution of
\begin{equation}\label{eq:GenSolRec}
D_N[f] = \sum_{r=1}^{k+2} c_r A_r^N
\end{equation}
where the $c_r$ are coefficients independent of $N$ (but do depend upon $A_1,\dots,A_{k+2}$).
\end{lemma}

\begin{proof}
    To find the general solution to a linear difference equation with constant coefficients, we need to find the roots of the generating polynomial, namely
    \[
    x^{k+2} - C_1 x^{k+1} + C_2 x^{k} - \dots + (-1)^{k} C_{k+2} = 0.
    \]
Recall that $C_1, \dots, C_{k+2}$ are the elementary symmetric polynomials in $k+2$ variables, and it is well-known (and simple to directly verify) that they satisfy
\[
\prod_{j=1}^{k+2} \left(x - A_j\right) = x^{k+2} - C_1 x^{k+1} + C_2 x^{k} - \dots + (-1)^{k} C_{k+2}.
\]
Therefore we have identified all $k+2$ roots of the generating polynomial for the linear recurrence as $x=A_r$, and it is well-known that if all the roots are distinct, then the general solution is a linear combination of the $N$\textsuperscript{th} power of each root.

This proves the general solution of the difference equation is as given in the statement of the lemma.
\end{proof}

\begin{lemma}\label{lem:RecurrenceSpecificSolution}
The specific solution of \eqref{eq:GenSolRec} which satisfies the initial conditions given in Lemma~\ref{lem:recurrence} has coefficients
\[
c_r =  A_r^{k+1} \prod_{\substack{j=1 \\ j \neq r}}^{k+2} \frac{1}{(A_r - A_j)}.
\]
\end{lemma}

\begin{proof}
Let
\begin{align*}
\widetilde D_N &:= \sum_{r=1}^{k+2} c_r A_r^N\\
&=\sum_{r=1}^{k+2} A_r^{N+k+1} \prod_{\substack{j=1 \\ j \neq r}}^{k+2} \frac{1}{(A_r - A_j)}
\end{align*}

We will show $\widetilde D_N$ equals the values of $D_N[f]$ given in Lemma~\ref{lem:recurrence} when $N$ runs from $1$ to $k+2$, and since it has the same shape as the general solution given in Lemma~\ref{lem:RecurrenceGeneralSolution}, this proves $\widetilde D_N = D_N[f]$ for all $N$.

First note that a direct application of equation~\eqref{eq:combsum} with $k+2$ variables (rather than $k$ variables in that equation) and with $n=N+k+2$ shows that
\begin{align}
\widetilde D_N &= \sum_{r=1}^{k+2} A_r^{N+k+1} \prod_{\substack{j=1 \\ j \neq r}}^{k+2} \frac{1}{(A_r - A_j)} \notag\\
&= \sum_{r=1}^{k+2} A_r^{N+k+1}  (-1)^{k+1} \prod_{\substack{j=1 \\ j \neq r}}^{k+2} \frac{1}{(A_j - A_r)} \notag\\
&= h_{N}(A_1,\dots,A_{k+2}) \label{eq:DtildeCompleteSymPoly}
\end{align}
where $h_N(A_1,\dots,A_{k+2})$ is the complete homogeneous symmetric polynomial of degree $N$ in $k+2$ variables.

Next recall the relationship between the elementary symmetric polynomials and the complete homogeneous ones, namely for $m=1,\dots,k+2$
\[
\sum_{j=0}^m (-1)^j C_j\left(A_1,\dots,A_{k+2}\right) h_{m-j}\left(A_1,\dots,A_{k+2}\right) = 0
\]
so when we substitute \eqref{eq:DtildeCompleteSymPoly}, this shows
\[
\sum_{j=0}^m (-1)^j C_j \widetilde D_{m-j} = 0
\]
which is exactly the initial conditions given in Corollary~\ref{lem:recurrence}.

Therefore we can conclude that $D_N[f] = \widetilde D_N$, since Lemma~\ref{lem:RecurrenceGeneralSolution} shows it satisfies the general solution of the linear recurrence relation given in Lemma~\ref{lem:recurrence}, and this lemma is just shown it satisfies all the initial conditions given in Lemma~\ref{lem:recurrence}. That is, we have proven that for any $N$,
\begin{equation}\label{eq:DNfClosedForm}
D_N[f] = (-1)^{k+1} \sum_{r=1}^{k+2}  A_r^{N+k+1} \prod_{\substack{j=1 \\ j \neq r}}^{k+2} \frac{1}{(A_j - A_r)}.
\end{equation}
\end{proof}

\begin{remark}
As an aside, recall \eqref{eq:ShiftedCharPolyInTermsOfDN} and \eqref{eq:f}. Setting $A_k = e^{-i\alpha_k}$ and $N \mapsto N-1$, a consequence of this proof and equation~\eqref{eq:DtildeCompleteSymPoly} is that
\[
\E_N\left[ \prod_{r=1}^k Z(\theta_N+\alpha_r) \right] =  \frac{1}{N}\prod_{r=1}^k \left(1-e^{-i\alpha_r}\right) h_{N-1}\left(e^{-i\alpha_1} , \dots, e^{-i\alpha_k} , 1,1 \right)
\]
This is one of the exact representations of the shifted discrete moments of the characteristic polynomial given in Theorem~\ref{thm:CharPoly}.
However, that representation is not amenable for large $N$-asymptotic analysis, so we will continue to work with \eqref{eq:DNfClosedForm}.
\end{remark}

\begin{lemma}
    In the limit $A_{k+1} \to A_{k+2} = 1$, we have
\begin{equation*}
D_N[f] = \sum_{r=1}^{k} A_r^{N+k+1} \frac{1}{(A_r-1)^2} \prod_{\substack{j=1 \\ j \neq r}}^{k} \frac{1}{(A_r - A_j)} + \left(N+k+1 - \sum_{\ell=1}^k \frac{1}{1-A_\ell} \right) \prod_{j=1}^k \frac{1}{1-A_j}
\end{equation*}
\end{lemma}

\begin{proof}
Starting with \eqref{eq:DNfClosedForm} we notice that for $r=1,\dots,k$ when $A_{k+1} \to A_{k+2} = 1$ we have
\[
 A_r^{N+k+1} \prod_{\substack{j=1 \\ j \neq r}}^{k+2} \frac{1}{(A_r - A_j)} \to  A_r^{N+k+1} \frac{1}{(A_r - 1)^2} \prod_{\substack{j=1 \\ j \neq r}}^{k} \frac{1}{(A_r - A_j)}
\]
and the $r=k+1$ and $r=k+2$ terms in the sum together equal, upon setting $A_{k+2}=1$
\begin{equation*}
 \frac{A_{k+1}^{N+k+1}}{A_{k+1}-1} \prod_{j=1}^{k} \frac{1}{(A_{k+1} - A_j)} + \frac{1}{1-A_{k+1}}\prod_{j=1}^{k} \frac{1}{(1 - A_j)} = \frac{ A_{k+1}^{N+k+1} \prod_{j=1}^{k} \frac{1}{(A_{k+1} - A_j)} - \prod_{j=1}^{k} \frac{1}{(1 - A_j)} }{A_{k+1}-1}
\end{equation*}
and an application of L'Hopital's rule shows as $A_{k+1} \to 1$ this equals
\[
\left(N+k+1 - \sum_{\ell=1}^k \frac{1}{1-A_\ell} \right) \prod_{j=1}^k \frac{1}{1-A_j}
\]
as required.
\end{proof}

\begin{corollary}\label{cor:shiftedMmtsExact}
    We have
\begin{equation*}
\E_N\left[ \prod_{r=1}^k Z(\theta_N+\alpha_r) \right]
= \frac{1}{N} \left(\sum_{\ell=1}^{k} \frac{e^{-i(N+k)\alpha_\ell} }{1-e^{-i\alpha_\ell}} \prod_{\substack{j=1 \\ j \neq \ell}}^{k} \frac{1-e^{-i\alpha_j}}{(e^{-i\alpha_\ell} - e^{-i\alpha_j})} + N+k - \sum_{\ell=1}^k \frac{1}{1-e^{-i\alpha_\ell}} \right).
\end{equation*}
\end{corollary}

\begin{proof}
Plug the previous lemma, with $N$ replaced by $N-1$ and $A_r$ replaced by $e^{-i\alpha_r}$ for $r=1,\dots,k$, into equation \eqref{eq:ShiftedCharPolyInTermsOfDN}. This is the one of the terms of the shifted discrete moments of the characteristic polynomial in Theorem~\ref{thm:CharPoly}.
\end{proof}

The following lemma is a restatement of Theorem~\ref{thm:CharPolyShiftsLeadingOrder}.
\begin{lemma}
As $N \to \infty$ uniformly for bounded $a_1,\dots,a_r$, we have
\begin{equation}\label{eq:CharPolyShiftedAsympt}
\lim_{N \to \infty} \E_N\left[ \prod_{r=1}^k Z\left(\theta_N+\frac{a_r}{N}\right) \right]
=   \sum_{m=0}^\infty \frac{(-1)^{m} i^{k+m}}{(m+k+1)!} h_{m}(a_1,\dots,a_k)   \prod_{j=1}^k a_j
\end{equation}
\end{lemma}

\begin{proof}
    Set $\alpha_r = a_r/N$ in Corollary~\ref{cor:shiftedMmtsExact}, take the first two terms in the  expansion
    \[
    e^{-ia_r / N} = 1 - \frac{i a_r}{N} + O\left(\frac{1}{N^2}\right)
    \]
    and let $N\to\infty$ to see that
\[
\lim_{N \to \infty} \E_N\left[ \prod_{r=1}^k Z\left(\theta_N+\frac{a_r}{N}\right) \right]
= -i \sum_{\ell=1}^{k} \frac{e^{-i a_\ell} }{a_\ell} \prod_{\substack{j=1 \\ j \neq \ell}}^{k} \frac{a_j}{(a_j-a_\ell)} + 1 + i \sum_{\ell=1}^k \frac{1}{a_\ell}
\]
Now expand the $e^{-i a_\ell}$ terms as an absolutely convergent series
\[
\lim_{N \to \infty} \E_N\left[ \prod_{r=1}^k Z\left(\theta_N+\frac{a_r}{N}\right) \right]
= -i \sum_{n=0}^\infty \frac{(-i)^n}{n!}\sum_{\ell=1}^{k} \frac{a_\ell^n }{a_\ell} \prod_{\substack{j=1 \\ j \neq \ell}}^{k} \frac{a_j}{(a_j-a_\ell)} + 1 + i \sum_{\ell=1}^k \frac{1}{a_\ell}.
\]
Lemma~\ref{LemmaProdSum}, specifically equation~\eqref{eq:LemmaZero}, shows that the terms with $n=0$ perfectly cancel the $\sum_{\ell=1}^k \frac{1}{a_\ell}$ term at the end.

Equation~\eqref{eq:lem:n0Term} of the same lemma shows that for the $n=1$ terms, we have
\[
(-i)^2 \sum_{\ell=1}^k \prod_{\substack{j=1 \\ j\neq \ell}}^k \frac{a_j}{(a_j-a_\ell)} = -1
\]
which perfectly cancels the middle term of $+1$.

Finally, Lemma~\ref{LemmaProdSum} shows that the terms with $n \leq k$ vanish, and for $n \geq k+1$ we have
\begin{equation*}
\frac{(-i)^{n+1}}{n!} \sum_{\ell=1}^k a_\ell^{n-1} \prod_{\substack{j=1 \\ j\neq \ell}}^k  \frac{a_j}{(a_j-a_\ell)} = \frac{(-1)^{k+n} i^{n+1}}{n!} h_{n-k-1}(a_1,\dots,a_k) \prod_{j=1}^k a_j.
\end{equation*}

Combining these pieces together, the only terms that survive are those with $n\geq k+1$ and we have
\begin{align*}
\lim_{N \to \infty} \E_N\left[ \prod_{r=1}^k Z\left(\theta_N+\frac{a_r}{N}\right) \right]
&=  \sum_{n=k+1}^\infty \frac{(-1)^{k+n} i^{n+1}}{n!} h_{n-k-1}(a_1,\dots,a_k) \prod_{j=1}^k a_j \\
&=  \sum_{m=0}^\infty \frac{(-1)^{m} i^{k+m}}{(m+k+1)!} h_{m}(a_1,\dots,a_k)   \prod_{j=1}^k a_j
\end{align*}
as requried.
\end{proof}

\begin{corollary}
    For $n_1,\dots,n_k \in \mathbb{N}$, we have
    \[
    \lim_{N\to\infty} \frac{1}{N^{n_1+\dots+n_k}} E_N\left[Z^{(n_1)}(\theta_N) \dots Z^{(n_k)}(\theta_N)\right] = (-1)^{k+n_1+\dots+n_k} i^{n_1+\dots+n_k} \frac{n_1! \dots n_k!}{(n_1+\dots+n_k+1)!}
    \]
\end{corollary}

\begin{proof}
On both sides of \eqref{eq:CharPolyShiftedAsympt} find the coefficient of $a_1^{n_1} \dots a_k^{n_k}$.

On the left-hand side that will be
\[
\lim_{N\to\infty} \E_N\left[ \prod_{r=1}^k \frac{1}{(n_r)!} Z^{(n_r)}(\theta_N) \frac{1}{N^{n_r}} \right]
\]
and on the right-hand side there will be exactly one term, which will be when $m=n_1+\dots+n_k-k$, with coefficient
\[
\frac{(-1)^{n_1+\dots+n_k-k} i^{n_1+\dots+n_k} }{(n_1+\dots+n_k+1)!}
\]
Combining these two pieces yields the corollary, which is the same statement as Theorem~\ref{thm:CharPolyDeriv}.
\end{proof}

\section{Deriving Conjecture~\ref{RatioMom}}\label{ProofConj3}
We are required to find
\begin{equation*}
    \sum_{0 < \gamma \leq T}  \zeta \left( \frac{1}{2} + i \gamma + \alpha_1 \right)\dots\zeta \left( \frac{1}{2} + i \gamma + \alpha_k \right)
\end{equation*}
where the $\alpha_j$ satisfy certain conditions as specified in Section~\ref{sect:Intro}.

We may use Cauchy's theorem to write
\[
   \sum_{0 < \gamma \leq T}  \zeta \left( \frac{1}{2} + i \gamma + \alpha_1 \right)\dots\zeta \left( \frac{1}{2} + i \gamma + \alpha_k \right) = \frac{1}{2 \pi i} \int_{R} \frac{\zeta '}{\zeta} (s) \zeta \left( s + \alpha_1 \right)\dots\zeta \left( s + \alpha_k \right) \ ds,
\]
where $R$ is a positively oriented rectangular contour with vertices $c+i$, $c+iT$, $1-c+iT$, and $1-c+i$, where $1/2 < c < 3/4$. For simplicity we assume that $|T-\gamma| \gg 1/\log T$ for any zero $\gamma$ and $T$ sufficiently large, although this constraint has no effect on the final answer.

Standard arguments show that the bottom, right, and top sides of the contour do not contribute anything except to the error term, with the largest contribution coming from the top piece of the contour at $O\left(T^{1/2+\varepsilon}\right)$.

We take the logarithmic derivative of the functional equation of $\zeta (s)$ 
\begin{equation*}
    \zeta (s) = \chi (s) \zeta (1-s)
\end{equation*}
to write
\[
\frac{\zeta '}{\zeta} (s) = \frac{\chi '}{\chi} (s) - \frac{\zeta '}{\zeta} (1-s).
\]
Using this we may rewrite the left-hand side of the contour as
\[
\frac{1}{2 \pi i} \int_{1-c+iT}^{1-c+i} \left( \frac{\chi '}{\chi} (s) - \frac{\zeta '}{\zeta} (1-s) \right) \zeta \left( s + \alpha_1 \right)\dots\zeta \left( s + \alpha_k \right) \ ds.
\]

For $t \geq 1$ and fixed $\sigma$, we have
\[
\frac{\chi '}{\chi}(s) = -\log \frac{t}{2\pi} + O\left( \frac{1}{t}\right),
\]
so this part of the integral contributes
\[
\frac{1}{2 \pi} \int_{1}^{T} \left( \log \frac{t}{2\pi} \right) \zeta \left( s + \alpha_1 \right)\dots\zeta \left( s + \alpha_k \right) \ ds
\]
(where we use the minus sign to flip the direction of the integral). Shifting the integral to the half-line at the cost of a small error term gives
\[
\frac{1}{2 \pi} \int_{1}^{T} \left( \log \frac{t}{2\pi} \right) \zeta \left( \frac{1}{2} + it + \alpha_1 \right)\dots\zeta \left( \frac{1}{2} + it + \alpha_k \right) \ ds + O\left( T^{1/2+\varepsilon} \right)
\]
and so upon taking the integral we have a contribution from this part of the left-hand side of the contour of
\[
\frac{T}{2\pi} \log \frac{T}{2\pi} + O\left( T^{1/2+\varepsilon} \right).
\]

Therefore we have shown that
\begin{multline}\label{eq:prodshiftedzeta}
\sum_{0 < \gamma \leq T}  \zeta \left( \frac{1}{2} + i \gamma + \alpha_1 \right)\dots\zeta \left( \frac{1}{2} + i \gamma + \alpha_k \right) \\
=  \frac{1}{2 \pi i} \int_{c+i}^{c+i T}\frac{\zeta '}{\zeta} (s) \zeta (1-s + \alpha_1) \dots \zeta (1-s + \alpha_k) \ ds + \frac{T}{2\pi} \log \frac{T}{2\pi} + O\left(T^{1/2+\varepsilon}\right).
\end{multline}

The remaining integral can be rewritten as
\begin{equation*}
\left. \frac{d}{d \delta} \frac{1}{2 \pi} \int_{1}^{T} \frac{\zeta (c+it + \delta)}{\zeta (c+it)}  \zeta (1-c-it + \alpha_1) \dots \zeta (1-c-it + \alpha_k) \ dt \right|_{\delta=0}
\end{equation*}
for some small shift $\delta$.

We now apply the recipe from the Ratios Conjecture, where the steps are outlined in Section \ref{sect:RatHis}. Step 1 tells us that for the $\zeta (s)$ term in the denominator, we replace it with its Dirichlet series
\[
\frac{1}{\zeta (s)} = \sum_{n=1}^{\infty} \frac{\mu (n)}{n^s}
\]
noting that, assuming the Riemann Hypothesis, this is conditionally convergent on the line of integration. Step 2 tells us that for any $\zeta (s)$ terms in the numerator, we replace them using the approximate functional equation
\[
\zeta (s) = \sum_{n \leq \sqrt{t/2 \pi}} \frac{1}{n^s} + \chi (s) \sum_{n \leq \sqrt{t/2 \pi}} \frac{1}{n^{1-s}}
\]
ignoring the error terms, obtaining
\begin{multline}\label{eq:shiftedzetaint}
	\frac{d}{d \delta} \frac{1}{2 \pi} \int_{1}^{T}   \sum_h \frac{\mu(h)}{h^{c+it}}  \left(\sum_m \frac{1}{m^{c+it+\delta}} + \chi (c+it+\delta) \sum_m \frac{1}{m^{1-c-it-\delta}}\right) \\ \times \left. \prod_{r=1}^k \left(\sum_{n_r} \frac{1}{n_r^{1-c-it+\alpha_r}} + \chi (1-c-it+\alpha_r) \sum_{n_r} \frac{1}{n_r^{c+it-\alpha_r}}\right)  \ dt \right|_{\delta=0}
\end{multline}

Step 3 of the Ratio's Conjecture tells us to now expand the numerator, yielding $2^{k+1}$ terms coming from picking either the sum without a $\chi$ factor or the sum with a $\chi$ factor in each zeta function.

To illustrate the next couple of steps in the Ratios Conjecture, we will concentrate on just one of the $2^{k+1}$ possible terms; we will consider the term were we keep every piece from each of the approximate functional equations without a $\chi (s)$ factor. This gives
\[
\sum_h \frac{\mu(h)}{h^{c+it}} \sum_m \frac{1}{m^{c+it+\delta}}  \sum_{n_1} \frac{1}{n_1^{1-c-it+\alpha_1}} \dots\sum_{n_k} \frac{1}{n_k^{1-c-it+\alpha_k}} .
\]
Step 4 tells us that when we integrate over $t$ we want to keep only the ``diagonal terms'', which are the terms in the sum where there are no oscillations in $t$, namely where
\[
hm = n_1\dots n_k.
\]
The ratios methodology assumes all the non-diagonal terms oscillate away, and so we are left with
\[
\sum_{hm = n_1\dots n_k} \frac{\mu(h)}{h^{1/2} m^{1/2+\delta} n_1^{1/2+\alpha_1} \dots n_k^{1/2+\alpha_k}}.
\]

Since this sum is multiplicative, we may write this as an Euler product. Writing
\[
h = p^b, m = p^d, n_j = p^{a_j}
\]
we have
\[
\prod_p \sum_{b+d = a_1 + \dots + a_k} \frac{\mu (p^b)}{p^{b/2} p^{d(1/2+\delta)} p^{a_1(1/2+\alpha_1)} \dots p^{a_k(1/2+\alpha_k)}}.
\]
We may only have $b=0$ (in which case $\mu(1) =1$), or $b=1$ (in which case $\mu(p) = -1$).

Summing over both possibilities for $b$ gives
\begin{multline*}
    \prod_p \sum_{a_1=0}^{\infty}\dots\sum_{a_k=0}^{\infty} \frac{1}{p^{a_1(1+\alpha_1+\delta)}} \dots \frac{1}{p^{a_k(1+\alpha_k+\delta)}} \\
    - \prod_p p^\delta \left( \sum_{a_1=0}^{\infty}\dots\sum_{a_k=0}^{\infty} \frac{1}{p^{a_1(1+\alpha_1+\delta)}} \dots \frac{1}{p^{a_k(1+\alpha_k+\delta)}} -1 \right).
\end{multline*}

Factorising  and using the sum of the geometric formula gives
\begin{equation*}
    \prod_p \frac{1}{1 - p^{-(1+\alpha_1+\delta)}} \dots \frac{1}{1-p^{-(1+\alpha_k+\delta)}} \left(1 - p^\delta( 1-(1-p^{-(1+\alpha_1+\delta)})\dots(1-p^{-(1+\alpha_k+\delta)}) \right).
\end{equation*}

The Euler product for $\zeta (s)$ allows us to rewrite this as
\begin{multline*}
\zeta (1+\alpha_1 + \delta)  \dots \zeta (1+\alpha_k +\delta) \\
\prod_p  \left( 1-p^{-(1+\alpha_1)} - \dots -p^{-(1+\alpha_k)} + p^{-(2+\alpha_1 + \alpha_2 + \delta)} + \dots+ p^{-(2+\alpha_{k-1} + \alpha_k + \delta)} \right. \\
\left.+\dots + (-1)^k p^{-(k + \alpha_1 +\dots+ \alpha_k +(k-1)\delta)} \right).
\end{multline*}

Finally, in Step 5, to make this product over primes converge we take out factors of $\zeta(1+ \alpha_j)$ for each $j=1,\dots,k$. In doing so, we have
\begin{equation*}
Z_{\alpha_1,\dots,\alpha_k,\delta} = \frac{  \zeta (1+\alpha_1 + \delta)  \dots \zeta (1+\alpha_k +\delta)}{\zeta (1+\alpha_1)  \dots \zeta (1+\alpha_k)}  \prod_p \frac{1 - p^{-(1+\alpha_1)} +\dots+ (-1)^k p^{-(k + \alpha_1 +\dots+ \alpha_k +(k-1)\delta)}}{(1-p^{-(1+\alpha_1)})\dots(1-p^{-(1+\alpha_k)})}.
\end{equation*}
For later simplicity, noticing the symmetry in the parameters $\alpha_1,\dots,\alpha_k$, we will denote the product over primes in the previous line by
\[
A_{\{\alpha_1,\dots,\alpha_k\}}(\delta) :=  \prod_p \frac{1 - p^{-(1+\alpha_1)} +\dots+ (-1)^k p^{-(k + \alpha_1 +\dots+ \alpha_k +(k-1)\delta)}}{(1-p^{-(1+\alpha_1)})\dots(1-p^{-(1+\alpha_k)})} .
\]

We now return to the numerator in \eqref{eq:shiftedzetaint} when the zeta terms have been replaced by the appropriate approximate functional equation.  The ratios methodology tells us to keep only the terms with an equal number of $\chi (s)$ and $\chi (1-s)$ terms, since all the other terms will be  highly oscillatory, and hence are assumed to not contribute to the final answer.

Other than the case when there are no $\chi$ terms (dealt with above), it is clear from \eqref{eq:shiftedzetaint} that the only such terms will come from the single $\chi(c+it+\delta)$ factor, and the $\chi(1-c-it+\alpha_j)$ factors for each $j$, giving us the ``one-swap'' cases. This also explains why there are no larger swaps as we don't have the appropriate number of $\chi$-factors to cancel.

Note that
\[
\chi (c+it+\delta) \chi (1-c-it+\alpha_j) = \left( \frac{t}{ 2\pi} \right)^{-\delta - \alpha_j} \left( 1 + O \left(\frac{1}{|t|} \right) \right)
\]
and the rest of the calculation (namely, the summing over the diagonal terms) proceeds as in the ``zero-swap'' case, given above. Thus the term arising from the swap of $\alpha_j$ with $\delta$ contributes
\[
\left( \frac{t}{2 \pi} \right)^{-\alpha_j - \delta} Z_{ \alpha_1,\dots, \alpha_{j-1}, -\delta, \alpha_{j+1},\dots,\alpha_k,-\alpha_j} .
\]

Inserting these results into \eqref{eq:prodshiftedzeta}, overall we have
\begin{multline*}
\sum_{0 < \gamma \leq T}  \zeta \left( \frac{1}{2} + i \gamma + \alpha_1 \right)\dots\zeta \left( \frac{1}{2} + i \gamma + \alpha_k \right) = \\
 \left. \frac{d}{d \delta} \frac{1}{2 \pi} \int_{1}^{T} Z_{\alpha_1,\dots,\alpha_k,\delta} + \left( \frac{t}{2 \pi} \right)^{-\alpha_1 - \delta} Z_{-\delta, \alpha_2,\dots,\alpha_k,-\alpha_1} +\dots + \left( \frac{t}{2 \pi} \right)^{-\alpha_k - \delta} Z_{ \alpha_1,\dots,-\delta, -\alpha_k} dt  \right|_{\delta=0} \\
+\frac{T}{2\pi}\log\frac{T}{2\pi} + O\left(T^{1/2 +\varepsilon}\right),
\end{multline*}
where
\[
 Z_{\alpha_1,\dots,\alpha_k,\delta} = \frac{  \zeta (1+\alpha_1 + \delta)  \dots \zeta (1+\alpha_k +\delta)}{\zeta (1+\alpha_1)  \dots \zeta (1+\alpha_k)} A_{\{\alpha_1,\dots,\alpha_k\}}(\delta),
\]
completing the derivation of Conjecture~\ref{RatioMom}.

\section{Recovering Conjectures \ref{MixMom} and \ref{kthMom}}\label{LeadingOrder}
We aim to find the coefficient of $\alpha_1^{n_1} \dots \alpha_k^{n_k}$, or rather its leading order in terms of the largest power of $L = \log\frac{t}{2\pi}$.

We proceed to pull the delta derivative inside the integral. Note that
\[
\left. \frac{d}{d \delta} Z_{\alpha_1,\dots,\alpha_k,\delta} \right|_{\delta=0} =  A'_{\{\alpha_1,\dots,\alpha_k\}}(0) + \sum_{j=1}^k \frac{\zeta'(1+\alpha_j)}{\zeta(1+\alpha_j)}.
\]
Due to the fact that $\left. \frac{1}{\zeta(1-\delta)} \right|_{\delta=0} = 0$ and $\lim_{\delta\to0} \frac{\zeta'}{\zeta^2}(1-\delta)= -1$ we have
\begin{align*}
&\left. \frac{d}{d \delta} \left(\frac{t}{2 \pi} \right)^{-\alpha_1-\delta} Z_{-\delta,\alpha_2,\dots,\alpha_k,-\alpha_1} \right|_{\delta=0} \\
&= -\left(\frac{t}{2 \pi} \right)^{-\alpha_1} Z_{0,\alpha_2,\dots,\alpha_k,-\alpha_1} -\left(\frac{t}{2 \pi} \right)^{-\alpha_1} \zeta(1-\alpha_1) A_{\{0,\alpha_2,\dots,\alpha_k\}}(-\alpha_1) \prod_{j=2}^k \frac{\zeta(1+\alpha_j - \alpha_1)}{\zeta(1+\alpha_j)} \\
&= -\left(\frac{t}{2 \pi} \right)^{-\alpha_1} \zeta(1-\alpha_1) A_{\{0,\alpha_2,\dots,\alpha_k\}}(-\alpha_1) \prod_{j=2}^k \frac{\zeta(1+\alpha_j - \alpha_1)}{\zeta(1+\alpha_j)}
\end{align*}
where the first term in the second line vanishes (again due to the fact that $\left. \frac{1}{\zeta(1-\delta)} \right|_{\delta=0} = 0$). Therefore, if we let
\begin{multline*}
    W_{\{\alpha_1,\dots,\alpha_k\}\setminus \{ \alpha_j \}}(\alpha_j,t) = \frac{\zeta'(1+\alpha_j)}{\zeta(1+\alpha_j)}  \\
    - \left(\frac{t}{2 \pi} \right)^{-\alpha_j} \zeta(1-\alpha_j) A_{\{\alpha_1,\dots \alpha_{j-1}, 0 , \alpha_{j+1}, \dots ,\alpha_k\}}(-\alpha_j) \prod_{\substack{ \ell = 1\\ \ell \neq j}}^k \frac{\zeta(1+\alpha_\ell - \alpha_j)}{\zeta(1+\alpha_\ell)}
\end{multline*}
then we see that the Ratios Conjecture yields
\begin{multline}\label{eq:RatConW}
\sum_{0 < \gamma \leq T}  \zeta \left( \frac{1}{2} + i \gamma + \alpha_1 \right)\dots\zeta \left( \frac{1}{2} + i \gamma + \alpha_k \right) = \\
= \frac{1}{2 \pi} \int_{1}^{T} \left( A'_{\{\alpha_1,\dots,\alpha_k\}}(0) + \sum_{j=1}^k W_{\{\alpha_1,\dots,\alpha_k\}\setminus \{ \alpha_j \}}(\alpha_j,t)\right) dt + \frac{T}{2\pi} \log \frac{T}{2\pi}
+ O\left(T^{1/2 +\varepsilon}\right).
\end{multline}

We will follow the Ratios methodology, which is to replace each arithmetic piece by its leading order, which in our case means
\[
A_{\{\alpha_1,\dots,\alpha_k\}}(0) \sim 1.
\]
We also replace each zeta function by its expansion about its pole, that is,
\[
\zeta (1+x) = \frac{1}{x} + O(1)
\]
and
\[
\frac{\zeta '}{\zeta} (1+x) = -\frac{1}{x} + O(1).
\]
Writing $L=\log\frac{t}{2\pi}$ we have
\[
\left(\frac{t}{2 \pi} \right)^{-\alpha_1} = 1 - L \alpha_1 + \frac{1}{2!} L^2 \alpha_1^2 + \dots.
\]

Combing these expansions, we have 
\begin{equation*}
W_{\{\alpha_1,\dots,\alpha_k\}\setminus \{ \alpha_\ell \}}(\alpha_\ell,t) \sim -\frac{1}{\alpha_\ell} + \left(1 - L \alpha_1 + \frac{1}{2!} L^2 \alpha_1^2 + \dots \right) \frac{1}{\alpha_\ell} \prod_{\substack{1\leq j \leq k \\ j \neq \ell}} \frac{\alpha_j}{\alpha_j-\alpha_\ell}.
\end{equation*}

Returning to \eqref{eq:RatConW}, we see that
\begin{multline*}
\sum_{0 < \gamma \leq T} \zeta \left( \frac{1}{2} + i \gamma + \alpha_1 \right)\dots\zeta \left( \frac{1}{2} + i \gamma + \alpha_k \right) \\
\sim  \frac{1}{2 \pi} \int_{1}^{T} \left( - \sum_{\ell=1}^k \frac{1}{\alpha_\ell} + \sum_{\ell=1}^k \left(1 - L \alpha_1 + \frac{1}{2!} L^2 \alpha_1^2 + \dots \right) \frac{1}{\alpha_\ell} \prod_{\substack{1\leq j \leq k \\ j \neq \ell}} \frac{\alpha_j}{\alpha_j-\alpha_\ell} \right) dt + \frac{T}{2\pi} \log\frac{T}{2\pi}.
\end{multline*}

In a very similar manner to the combinatorial arguments given in Section~\ref{sect:RMTMot}, we can see by Lemma \ref{LemmaProdSum}, equation~\eqref{eq:LemmaZero}, that the poles from the sums over $\ell$, where we take the $1$-term in the bracket containing the logarithms, cancel perfectly. This means that we have to evaluate
\begin{align*}
\sum_{0 < \gamma \leq T}  \zeta \left( \frac{1}{2} + i \gamma + \alpha_1 \right)&\dots\zeta \left( \frac{1}{2} + i \gamma + \alpha_k \right) \\
\sim  \frac{1}{2 \pi}& \int_{1}^{T} \sum_{\ell=1}^k \left( - L \alpha_1 + \frac{1}{2!} L^2 \alpha_1^2 + \dots \right) \frac{1}{\alpha_\ell} \prod_{\substack{1\leq j \leq k \\ j \neq \ell}} \frac{\alpha_j}{\alpha_j-\alpha_\ell} \ dt + \frac{T}{2\pi} \log\frac{T}{2\pi}.
\end{align*}

To find the coefficient of the $\alpha_\ell^n$ term, say, we need to take the $(-1)^{n+1}\alpha_\ell^{n+1}/(n+1)!$ term for each $\alpha_\ell$ from the bracket containing the logarithms, as one $\alpha_\ell$ will cancel with the one in the denominator for each $\ell$. Summing over all such $\ell$, we have
\begin{equation*}
\sum_{0 < \gamma \leq T}  \zeta \left( \frac{1}{2} + i \gamma + \alpha_1 \right)\dots\zeta \left( \frac{1}{2} + i \gamma + \alpha_k \right)
\sim \frac{1}{2 \pi} \int_{1}^{T} \left( \frac{(-1)^{n+1}}{(n+1)!} L^{n+1} \right) \sum_{\ell=1}^k \alpha_\ell^n \prod_{ \substack{j=1 \\ j \neq \ell}}^{k} \frac{\alpha_j}{\alpha_j-\alpha_\ell} \ dt.
\end{equation*}

By Lemma \ref{LemmaProdSum}, for $n \geq k$, this is
\begin{multline*}
\sum_{0 < \gamma \leq T}  \zeta \left( \frac{1}{2} + i \gamma + \alpha_1 \right)\dots\zeta \left( \frac{1}{2} + i \gamma + \alpha_k \right) \\
\sim \frac{1}{2 \pi} \int_{1}^{T} \left( \frac{(-1)^{n+1}}{(n+1)!} L^{n+1} \right) (-1)^{k+1} h_{n-k}(\alpha_1,\dots,\alpha_k) \prod_{j=1}^k \alpha_j \ dt.
\end{multline*}

Since $n=n_1+\dots +n_k$, and writing $ h_{n-k}(\alpha_1,\dots,\alpha_k) \prod_{j=1}^k \alpha_j $ out as a sum, we have
\begin{multline*}
\sum_{0 < \gamma \leq T}  \zeta \left( \frac{1}{2} + i \gamma + \alpha_1 \right)\dots\zeta \left( \frac{1}{2} + i \gamma + \alpha_k \right) \\
\sim \frac{1}{2 \pi} \int_{1}^{T} \left( \frac{(-1)^{n_1+\dots n_k+k}}{(n_1+\dots n_k+1)!} L^{n_1+\dots n_k+1} \right) \sum_{\substack{n_1 \geq 1,\dots,n_k \geq 1 \\ n_1 + \dots + n_k = n}} \alpha_1^{n_1} \alpha_2^{n_2} \dots \alpha_k^{n_{k}} \ dt.
\end{multline*}

Finally, Taylor expand the left-hand side around each $\alpha_\ell$, differentiate with respect of each $\alpha_\ell$ $n_\ell$-times and set $\alpha_\ell=0$. Multiplying through by $n_\ell!$ for each $\ell$ gives
\begin{multline*}
\sum_{0 < \gamma \leq T} \zeta^{(n_1)} \left( \frac{1}{2} + i \gamma \right) \dots \zeta^{(n_k)} \left( \frac{1}{2} + i \gamma \right) \\
\sim \frac{1}{2 \pi} \int_{1}^{T} \left( (-1)^{n_1+\dots n_k+k} \frac{n_1!\dots n_k!}{(n_1+\dots n_k+1)!} L^{n_1+\dots n_k+1} \right) \ dt.
\end{multline*}

Integrating gives Conjecture \ref{MixMom} and additionally, setting all $n_\ell$ equal to $n$ gives Conjecture \ref{kthMom}.

\section{Examples}\label{Examples}
\subsection{Recovering the Shanks' conjecture asymptotic ($k=1$)} \label{ShanksInt} \hfill

For the case $k=1$ we will recover the full asymptotic for the Generalised Shanks' Conjecture in an integral form, including the lower order terms, which is equivalent to the asymptotic found in \cite{HugPC22} for all derivatives.

By \eqref{eq:RatConW} we can write
\begin{equation}\label{eq:Shanks}
\sum_{0 < \gamma \leq T}  \zeta \left(\frac{1}{2} + i \gamma + \alpha \right) = \frac{1}{2 \pi} \int_{1}^{T} \left( A'_{\{\alpha\}}(0) +W(\alpha,t) \right) \ dt + \frac{T}{2\pi} \log\frac{T}{2\pi} + O\left(T^{1/2 +\varepsilon}\right)
\end{equation}
where
\[
W(\alpha,t) = \frac{\zeta'(1+\alpha)}{\zeta(1+\alpha)}  - \left(\frac{t}{2 \pi} \right)^{-\alpha} \zeta(1-\alpha) A_{\{0\}}(-\alpha).
\]
Now by \eqref{eq:A}, we see that
\[
A_{\{\alpha\}}(0)= \prod_p \frac{(1-p^{-(1+\alpha)})}{(1-p^{-(1+\alpha)})} = 1
\]
and so $ A'_{\{\alpha\}}(0)=0$. Then \eqref{eq:Shanks} can be simplified to
\begin{equation}\label{eq:Shanks2}
    \sum_{0 < \gamma \leq T}  \zeta \left(\frac{1}{2} + i \gamma + \alpha \right) = \frac{1}{2 \pi} \int_{1}^{T} \left( \frac{\zeta'(1+\alpha)}{\zeta(1+\alpha)}  - \left(\frac{t}{2 \pi} \right)^{-\alpha} \zeta(1-\alpha) \right) \  dt + \frac{T}{2\pi} \log\frac{T}{2\pi} + O\left(T^{1/2 +\varepsilon}\right).
\end{equation}

The Laurent expansion for $\frac{\zeta' (1+\alpha)}{\zeta (1+\alpha)}$ about $\alpha = 0$ is given by
\[
\frac{\zeta' (1+\alpha)}{\zeta (1+\alpha)} = - \frac{1}{\alpha} + \sum_{j=0}^{\infty} A_j \alpha^j,
\]
the Laurent expansion for $-\zeta (1-\alpha)$ about $\alpha=0$ gives
\[
- \zeta (1-\alpha) = \frac{1}{\alpha} - \sum_{j=0}^{\infty} \frac{\gamma_j}{j!} \alpha^j,
\]
and the Taylor expansion of $(t/2 \pi)^{-\alpha}$ about $\alpha=0$ gives
\[
\left( \frac{t}{2 \pi} \right)^{-\alpha} = \sum_{j=0}^{\infty} \frac{(-1)^j \alpha^j L^j}{j!},
\]
where $L=\log (t/2\pi)$, the $\gamma_j$ are the Stieltjes coefficients from the expansion of $\zeta (s)$ about $s=1$, and $A_n$ is the $n$\textsuperscript{th} coefficient from the expansion of $\zeta '(s)/\zeta (s)$ about $s=1$.

Substitute these series into \eqref{eq:Shanks2} and simplify to give
\begin{align*}
    \sum_{0 < \gamma \leq T}&  \zeta \left(\frac{1}{2} + i \gamma + \alpha \right) \\
    =\frac{1}{2 \pi}& \int_{1}^{T} \left( - \frac{1}{\alpha} + \sum_{j=0}^{\infty} A_j \alpha^j \right) + \left( \sum_{j=0}^{\infty} \frac{(-1)^j \alpha^j L^j}{j!} \right) \left( \frac{1}{\alpha} - \sum_{j=0}^{\infty} \frac{\gamma_j}{j!} \alpha^j \right)  \ dt + \frac{T}{2\pi} \log\frac{T}{2\pi} + O\left(T^{1/2 +\varepsilon}\right) \\
    =\frac{1}{2 \pi}& \int_{1}^{T} \sum_{j=0}^{\infty} A_j \alpha^j + \sum_{j=1}^{\infty} \frac{(-1)^j L^j \alpha^{j-1}}{j!} + \sum_{j=0}^{\infty} \sum_{m=0}^{j} \frac{(-1)^{m+1} L^m \gamma_{j-m} \alpha^j}{m! (j-m)!} \ dt + \frac{T}{2\pi} \log\frac{T}{2\pi} + O\left(T^{1/2 +\varepsilon}\right).
\end{align*}

For the general $n$\textsuperscript{th} derivative of $\zeta (s)$, where $n \in \mathbb{N}$, we have after Taylor expanding $\zeta (\rho + \alpha)$ about $\alpha =0$, differentiating $n$ times, setting $\alpha=0$, and multiplying through by $n!$, the following formula
\[
\sum_{0 < \gamma \leq T} \zeta ^{(n)} \left(\frac{1}{2} + i \gamma \right) = \frac{n!}{2 \pi} \int_{1}^{T} A_n + \frac{(-1)^{n+1} L^{n+1}}{(n+1)!} + \sum_{m=0}^{n} \frac{(-1)^{m+1} L^m \gamma_{n-m}}{m! (n-m)!} \ dt + O\left(T^{1/2 +\varepsilon}\right),
\]
which is equivalent to the main result from Hughes and Pearce-Crump~\cite{HugPC22} (which can be seen by integrating the above equation).

\subsection{The second moment of $\zeta ' (1/2 + i\gamma)$ $(k=2)$} \label{2ndMom} \hfill

To calculate the full asymptotic for the second moment of the first derivative of the Riemann zeta function, begin with \eqref{eq:RatConW} with two shifts $\alpha,\beta$ to write
\begin{multline*}
\sum_{0 < \gamma \leq T}  \zeta \left(\frac{1}{2} + i \gamma + \alpha \right) \zeta \left( \frac{1}{2} + i \gamma + \beta \right) \\
= \frac{1}{2 \pi} \int_{1}^{T} \left( A'_{\{\alpha,\beta\}}(0) +  \frac{\zeta'(1+\alpha)}{\zeta(1+\alpha)} - \left(\frac{t}{2 \pi} \right)^{-\alpha} \zeta(1-\alpha) A_{\{0,\beta\}}(-\alpha) \frac{\zeta(1+\beta - \alpha)}{\zeta(1+\beta)} \right. \\
\hphantom{spacespacespacespa} \left. + \frac{\zeta'(1+\beta)}{\zeta(1+\beta)} - \left(\frac{t}{2 \pi} \right)^{-\beta}\zeta(1-\beta) A_{\{\alpha,0\}}(-\beta) \frac{\zeta(1+\alpha - \beta)}{\zeta(1+\alpha)}  \right) \ dt \\
+ \frac{T}{2\pi} \log\frac{T}{2\pi} + O\left(T^{1/2 +\varepsilon}\right).
\end{multline*}

Note that by \eqref{eq:A} we have that the arithmetic part can be written as
\[
A_{\{\alpha, \beta\}}(\delta) = \prod_p \frac{1-p^{-(1+\alpha)}-p^{-(1+\beta)}+p^{-(2+\alpha+\beta+\delta)}}{\left(1-p^{-(1+\alpha)} \right) \left(1-p^{-(1+\beta)} \right)}.
\]
By adding `zero', we have
\begin{align}
    A_{\{\alpha, \beta\}}(\delta) &= \prod_p \frac{1-p^{-(1+\alpha)}-p^{-(1+\beta)}+p^{-(2+\alpha+\beta)}}{\left(1-p^{-(1+\alpha)} \right) \left(1-p^{-(1+\beta)} \right)} + \prod_p \frac{p^{-(2+\alpha+\beta+\delta)}-p^{-(2+\alpha+\beta)}}{\left(1-p^{-(1+\alpha)} \right) \left(1-p^{-(1+\beta)} \right)} \notag \\
    &= 1 + \prod_p \frac{p^{-(2+\alpha+\beta+\delta)}-p^{-(2+\alpha+\beta)}}{\left(1-p^{-(1+\alpha)} \right) \left(1-p^{-(1+\beta)} \right)} \label{eq:Arith2ndMom}
\end{align}
which is the in a nicer form to work with.

We now employ a computer package to perform the expansions. We substitute the Laurent expansion for $\zeta (s)$ about $s=1$, which is where the various Stieltjes constants will come from. 

Then for $L = \log \frac{t}{2\pi}$ and $\rho=1/2 + i\gamma$ a non-trivial zero of the Riemann zeta function,
\begin{align}
    &\sum_{0 < \gamma \leq T} \zeta ' \left(\frac{1}{2} + i \gamma \right)^2 = \notag \\
    &\hphantom{spa}\frac{1}{2 \pi} \int_{1}^{T} \Bigg( \frac{1}{6} L^3 A^{(0,0,0)} +\frac{1}{2} L^2 \left(2 \gamma_0 A^{(0,0,0)}+ A^{(0,0,1)}+A^{(0,1,0)}\right) \notag \\
   &\hphantom{spac}+\frac{1}{2} L \left(-8 \gamma_1 A^{(0,0,0)} + 4 \gamma_0 A^{(0,0,1)}+A^{(0,0,2)}+4 \gamma_0  A^{(0,1,0)}+2A^{(0,1,1)}-A^{(0,2,0)} \right) \notag \\
   &\hphantom{spaces}+\frac{1}{6}
   \left(-12 \gamma_0^3 A^{(0,0,0)} -36 \gamma_0 \gamma_1 A^{(0,0,0)} +6 \gamma_2 A^{(0,0,0)} -24 \gamma_1 A^{(0,0,1)}   \right. \notag \\
   &\hphantom{spacespa} +6 \gamma_0  A^{(0,0,2)} +A^{(0,0,3)} +12 \gamma_0  A^{(0,1,1)} +3 A^{(0,1,2)}-3 A^{(0,2,1)}+6 A^{(1,1,1)}  \notag \\
   &\hphantom{spacespace} -12 \gamma_1 A^{(0,1,0)} -12 \gamma_1 A^{(1,0,0)} +6 \gamma_0^2 A^{(0,1,0)} -6\gamma_0  A^{(0,2,0)} +A^{(0,3,0)} -3 A^{(2,1,0)} \notag \\
   &\hphantom{spacespacesp}  \left.  -6 \gamma_0^2 A^{(1,0,0)} +12 \gamma_0 A^{(1,1,0)}-3A^{(1,2,0)}\right) \Bigg) \ dt   + O\left( T^{1/2 +\varepsilon} \right)
\end{align}
where the $\gamma_m$ are the coefficients in the Laurent expansion of $\zeta (s)$ about $s=1$ and the $A^{(i,j,k)}$ terms are various products over primes, where we are using the notation
\[
A^{(i,j,k)} = \frac{\partial^i}{\partial \alpha^i} \frac{\partial^j}{\partial \beta^j}\frac{\partial^k}{\partial \delta^k} A_{\{\alpha, \beta\}}(\delta)\Bigg|_{\alpha=\beta=\delta=0}
\]
for complex numbers (shifts) $\alpha, \beta, \delta$ satisfying the conditions of Conjecture~\ref{RatioMom}.

We now simplify this result to state a final version of the conjecture for this moment. Clearly, say from \eqref{eq:Arith2ndMom}, if we do not differentiate with respect to $\delta$ and set $\delta=0$, then the arithmetic piece equals $1$. If we then differentiate with respect to $\alpha$ or $\beta$, and then set these equal to $0$, these terms disappear. That is, for $i,j >0$,
\[
A^{(i,j,0)} = 0,
\]
while if we don't differentiate at all (that is, $i=j=k=0$) and set $\alpha,\beta,\delta=0$, we have
$$A^{(0,0,0)}=1.$$
It is also clear that the arithmetic piece is symmetric in the first two derivatives, so
\[
A^{(i,j,k)} = A^{(j,i,k)}.
\]
This means that whenever we see these terms, we can combine them together which is why the conjecture does not appear symmetric in $i,j,k$.

After using these rules for $A^{(i,j,k)}$, we form the conjecture for the second moment of the first derivative of the Riemann zeta function, initially stated as Conjecture \ref{conj:secmom} in this paper and restated here for convenience.

\begin{conjecture*}
    Assume the Riemann Hypothesis. For $L = \log \frac{t}{2\pi}$ and $\rho=1/2 + i\gamma$ a non-trivial zero of the Riemann zeta function,
\begin{multline*}
\sum_{0 < \gamma \leq T} \zeta ' \left(\frac{1}{2} + i \gamma \right)^2 = \\
\frac{1}{2 \pi} \int_{1}^{T} \Big( \frac{1}{6}L^3 + \frac{1}{2}L^2 \left(2\gamma_0 +A^{(0,0,1)}  \right)  + \frac{1}{2} L \left( -8\gamma_1 + 4\gamma_0 A^{(0,0,1)} + A^{(0,0,2)}  + 2A^{(0,1,1)} \right) \\
+ \frac{1}{6} \left( -12\gamma_0^3 - 36\gamma_0 \gamma_1 + 6\gamma_2 -24 \gamma_1 A^{(0,0,1)} + 6 \gamma_0 A^{(0,0,2)} + A^{(0,0,3)}  \right. \\
  + 12 \gamma_0 A^{(0,1,1)} + 3 A^{(0,1,2)} -3 A^{(0,2,1)} +6 A^{(1,1,1)}  \Big) \ dt + O\left( T^{1/2 +\varepsilon} \right)
\end{multline*}
where the $\gamma_m$ are the coefficients in the Laurent expansion of $\zeta (s)$ about $s=1$ and the $A^{(i,j,k)}$ are arithmetic terms that are various products and sums over primes.
\end{conjecture*}

We present some numerical evidence. To generate these numerics, we again use a computer package. Most of the calculations are standard so we just describe how to calculate the arithmetic pieces and so the full polynomial that appears in the integral version of this moment.

We begin by taking logarithms of the arithmetic pieces $A_{\{\alpha, \beta\}}(\delta)$ and expand it as a series about $\alpha,\beta,\delta=0$, which we can do as they are all small by assumption. Note that since the highest derivative that we need to take is the third derivative for any of $\alpha,\beta,\delta$, we don't need to expand beyond the third powers. We then sum this expansion for the first $1000$ prime numbers (note that for higher accuracy we could have done this calculation for more primes but at a higher computational cost - this would've generated better numerics and would be of future interest).

In doing these calculations, we find that
\begin{equation*}
\sum_{0 < \gamma \leq T} \zeta ' \left(\frac{1}{2} + i \gamma \right)^2 = \frac{1}{2 \pi} \int_{1}^{T} \left( \frac{1}{6}L^3 - 0.03621 L^2 + 2.12487 L  -2.52789 \right) dt + O\left( T^{1/2 +\varepsilon} \right).
\end{equation*}

After integrating, we have
\begin{multline*}
\sum_{0 < \gamma \leq T} \zeta ' \left(\frac{1}{2} + i \gamma \right)^2 = \frac{1}{6} \frac{T}{2\pi} \left( \log \frac{T}{2\pi}  \right)^3 -0.52037 \frac{T}{2\pi} \left( \log \frac{T}{2\pi}  \right)^2 \\+ 2.95321 \frac{T}{2\pi} \left( \log \frac{T}{2\pi}  \right) -4.65238 \frac{T}{2 \pi} + O\left( T^{1/2 +\varepsilon} \right).
\end{multline*}

We are now in a place to compare our theoretical conjecture with the true numerics of the real part of $\sum_{0 < \gamma \leq T} \zeta ' \left(1/2 + i \gamma \right)^2$, where we ignore the imaginary part as it is small. Note that this mimic Shanks' Conjecture, and we claim that $\zeta'(\rho)^2$ is also real and positive on average. (Similar conjectures - up to the sign - can be made for all moments of mixed derivatives in this paper, which is clear by looking at the leading order asymptotic in Conjecture~\ref{MixMom}.)

We begin by plotting the real part of the cumulative total of the sum $\sum_{0 < \gamma \leq T} \zeta'\left(1/2+i\gamma\right)^2$ for the first 1,000,000 zeros, given in Figure \ref{fig:PlotZetaPrimeRowSquared}. Numerically we can show that the imaginary part is small and so don't plot this in the following figures (which agrees with our conjecture only containing real terms in the asymptotic).
\begin{figure}[hbt!]
\begin{center}
\includegraphics[height=6cm]{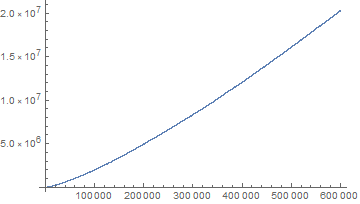}
\caption{The real part of the true value of $\sum_{0 < \gamma \leq T} \zeta' (1/2 + i \gamma)^2$, for $T$ up to the height of the 1,000,000\textsuperscript{th} zero.}\label{fig:PlotZetaPrimeRowSquared}
\end{center}
\end{figure}

In Figure \ref{fig:PlotZetaPrimeRowSquared_Subtract_L3} we have subtracted the main term, which shows a decrease in size of an order of magnitude, and also shows a sign change suggesting that we have the correct leading order behaviour.

\begin{figure}[hbt!]
\begin{center}
\includegraphics[height=6cm]{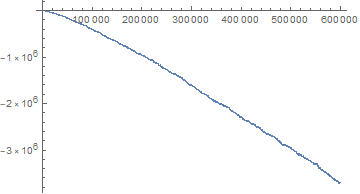}
\caption{The difference in the real part of the actual value of $\sum_{0 < \gamma \leq T} \zeta' (1/2 + i \gamma)^2$ and the leading asymptotic result of the equation, for $T$ up to the height of the 1,000,000\textsuperscript{th} zero.}\label{fig:PlotZetaPrimeRowSquared_Subtract_L3}
\end{center}
\end{figure}

Finally, in Figure \ref{fig:PlotZetaPrimeRowSquared_Subtract_L0} we plot the real part of the sum over zeros minus all the remaining terms in the asymptotic expansion, leaving the error term. The arithmetic pieces are calculated for the first $1000$ primes. Clearly there is excellent agreement, with a very small error compared with the original cumulative sum. 

\begin{figure}[hbt!]
\begin{center}
\includegraphics[height=6cm]{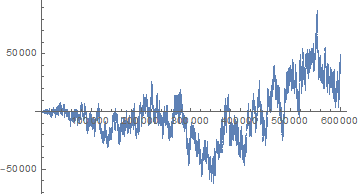}
\caption{The difference in the real part of the actual value of $\sum_{0 < \gamma \leq T} \zeta' (1/2 + i \gamma)^2$ and the whole asymptotic result of the equation, for $T$ up to the height of the 1,000,000\textsuperscript{th} zero, showing the real error at each point.}\label{fig:PlotZetaPrimeRowSquared_Subtract_L0}
\end{center}
\end{figure}

\begin{remark}
    As we have already remarked in this paper, the error term here could plausibly be of size $O\left(T^{1 - \delta}\right)$ for $0 < \delta < 1/2$, rather than of size $O\left(T^{1/2 +\varepsilon}\right)$. While we do not rule out the possibility of the first option, as discussed earlier, we have settled on this second option, in keeping with the original statements of the recipe/Ratios Conjecture.
\end{remark}

\section*{Acknowledgements}
We would like to thank Winston Heap and Junxian Li for their helpful comments. The majority of this paper constitutes the last chapter of the second author's PhD thesis \cite{theAPC}. The second author is now supported by the Heilbronn Institute for Mathematical Research where further edits to this papers has occured.

\bibliographystyle{abbrv}

\bibliography{bibliography}

\end{document}